\documentclass[11pt]{amsart}
\usepackage{amsmath,latexsym,amsfonts,amssymb}
\usepackage{mathrsfs}
\usepackage{fullpage}
\usepackage{enumerate}
\usepackage{amssymb, amsmath, amsfonts}
\usepackage{mathrsfs}
\usepackage{epsfig,amsbsy,amsthm} 
\usepackage{float,epsfig}
\usepackage{fixmath}
\usepackage{graphics}
\usepackage{pgfpages}
\usepackage{caption}
\usepackage{subcaption}
\usepackage{appendix}
\numberwithin{equation}{section}  
\usepackage{hyperref} 
\usepackage{graphicx}
\usepackage{pst-all}
\usepackage{calligra}
\usepackage{color}
\usepackage{tikz}
\DeclareMathAlphabet{\mathpzc}{OT1}{pzc}{m}{it}
\DeclareMathAlphabet{\mathcalligra}{T1}{calligra}{m}{n}
\begin{document}
\newtheorem{theorem}{\bf Theorem}[section]
\newtheorem{proposition}[theorem]{\bf Proposition}
\newtheorem{definition}{\bf Definition}[section]
\newtheorem{corollary}[theorem]{\bf Corollary}
\newtheorem{exam}[theorem]{\bf Example}
\newtheorem{remark}[theorem]{\bf Remark}
\newtheorem{lemma}[theorem]{\bf Lemma}
\newtheorem{assum}[theorem]{\bf Assumption}

\newcommand{\von}{\vskip 1ex}
\newcommand{\vone}{\vskip 2ex}
\newcommand{\vtwo}{\vskip 4ex}
\newcommand{\ds}{\displaystyle}
\def \noin{\noindent}
\newcommand{\be}{\begin{equation}}
\newcommand{\ee}{\end{equation}}
\newcommand{\beno}{\begin{equation*}}
\newcommand{\eeno}{\end{equation*}}
\newcommand{\ba}{\begin{align}}
\newcommand{\ea}{\end{align}}
\newcommand{\bano}{\begin{align*}}
\newcommand{\eano}{\end{align*}}
\newcommand{\bea}{\begin{eqnarray}}
\newcommand{\eea}{\end{eqnarray}}
\newcommand{\beano}{\begin{eqnarray*}}
\newcommand{\eeano}{\end{eqnarray*}}
\def \noin{\noindent}
 \def \tcK{{\tilde {\mathcal K}}}    
\def \O{{\Omega}}
\def \cT{{\mathcal T}}
\def \cV{{\mathcal V}}
\def \cE{{\mathcal E}}
\def \R{{\mathbb R}}
\def \V{{\mathbb V}}
\def \S{{\mathbb S}}
\def \N{{\mathbb N}}
\def \Z{{\mathbb Z}}
\def \Mc{{\mathcal M}}
\def \Cc{{\mathcal C}}
\def \Rc{{\mathcal R}}
\def \Ec{{\mathcal E}}
\def \Gc{{\mathcal G}}
\def \Tc{{\mathcal T}}
\def \Qc{{\mathcal Q}}
\def \Ic{{\mathcal I}}
\def \Pc{{\mathcal P}}
\def \Oc{{\mathcal O}}
\def \Uc{{\mathcal U}}
\def \Yc{{\mathcal Y}}
\def \Ac{{\mathcal A}}
\def \Bc{{\mathcal B}}
\def \k{\mathpzc{k}}
\def \Rp{\mathpzc{R}}
\def \Os{\mathscr{O}}
\def \Js{\mathscr{J}}
\def \Es{\mathscr{E}}
\def \Qs{\mathscr{Q}}
\def \Ss{\mathscr{S}}
\def \Cs{\mathscr{C}}
\def \Ds{\mathscr{D}}
\def \Ms{\mathscr{M}}
\def \Ts{\mathscr{T}}
\def \LL{L^{\infty}(L^{2}(\Omega))}
\def \LH{L^{2}(0,T;H^{1}(\Omega))}
\def \B {\mathrm{BDF}}
\def \el {\mathrm{el}}
\def \re {\mathrm{re}}
\def \e {\mathrm{e}}
\def \div {\mathrm{div}}
\def \CN {\mathrm{CN}}
\def \Rs   {\mathbf{R}_{{\mathrm es}}}
\def \Rb {\mathbf{R}}
\def \Jb {\mathbf{J}}
\def  \apos {\emph{a posteriori~}}

\def\mean#1{\left\{\hskip -5pt\left\{#1\right\}\hskip -5pt\right\}}
\def\jump#1{\left[\hskip -3.5pt\left[#1\right]\hskip -3.5pt\right]}
\def\smean#1{\{\hskip -3pt\{#1\}\hskip -3pt\}}
\def\sjump#1{[\hskip -1.5pt[#1]\hskip -1.5pt]}
\def\jumptwo{\jump{\frac{\p^2 u_h}{\p n^2}}}

\title[Adaptive FEM for optimal control problem with integral state constraints]
{Adaptive finite element method for an elliptic optimal control problem with integral state constraints}

\author{ Kamana Porwal}
\thanks{The first author's work is supported by CSIR Extramural Research Grant (Grant No. 25(0297)/19/EMR-II)}
\address{Department of Mathematics, Indian Institute of Technology Delhi,
New Delhi-110016}
\email{kamana@maths.iitd.ac.in}

\author{ Pratibha Shakya}
\address{Department of Mathematics, Indian Institute of Technology Delhi,
New Delhi-110016}
\email{shakya.pratibha10@gmail.com}

%
\date{}
\maketitle

\begin{abstract}
\par \noindent
In this article, we develop a posteriori error analysis of a nonconforming finite element method for a linear quadratic elliptic distributed optimal control problem with two different set of constraints, namely (i)  integral state constraint and integral control constraint (ii)  integral state constraint and pointwise control constraints. In the analysis, we have taken the approach of reducing the state-control constrained minimization problem into a state minimization problem obtained by eliminating the control variable. The reliability and efficiency of a posteriori error estimator are discussed. Numerical results are reported to illustrate the behavior of the error estimator.

\end{abstract}
\section{Introduction} \label{sec:intro}
Optimal control problems (OCPs) play an important role in various applications in physics, mechanics and other engineering sciences. For the theoretical  and numerical development of the OCPs, we refer to \cite{fredi10,Lions:1971:OCP,liuyan2008,HPU:2009:Appls}. The finite element method is a popular and widely used numerical method to approximate OCPs. The finite element approximation of the elliptic optimal control problems started with articles of Falk \cite{Falk:1973:OCP} and Geveci \cite{Geveci:1979:OCP}. In these papers, piecewise constant approximation of the control is considered and optimal order error estimates are obtained for the optimal variables.
The authors of \cite{arnautu} have established the optimality conditions and introduce the Ritz-Galerkin discretization for elliptic optimal control problem and obtained error estimates for the control and state variables. The authors of \cite{hinze:2005} have introduced the variational discretization method, therein the error estimates are obtained by exploiting the relationship between the state and adjoint state. The numerical approximation of the elliptic optimal control problems with control variable from measure spaces can be found in \cite{casas:2012:sparse,clason:2011}. \\
\par \noindent
There have been abundant research on the adaptive finite method for OCPs governed by differential equations in last few decades. The use of adaptive techniques based on a posteriori error estimation is well accepted in the context of finite element discretization of partial differential equations \cite{ainsworth:2000,verfurth:1995}. In this direction, the pioneer  work has been made by Liu and Yan \cite{liuyan2001} for residual based a posteriori error estimates,  and Becker et al. \cite{becker:2000} for dual-weighted goal oriented adaptivity for optimal control problems. In \cite{kohls:2015} the authors have proved that the sequence of adaptively generated discrete solutions converge to the true solutions of OCPs. Recently, Gong and Yan \cite{gong:2017} have presented a rigorous proof for convergence and quasi-optimality of adaptive finite element method for an OCP with pointwise control constraints by means of variational discretization technique. In \cite{wolfmayr:2016}, Wolfmayr  has derived functional type  a posteriori error estimates for elliptic optimal control problems with control constraints. The authors of \cite{pshakya:2019} have studied the finite element approximation of OCPs governed by elliptic equations with measure data, therein they have 
derived both a priori  and a posteriori error bounds for the state and control variables. We refer to the reference section for other notable works on the adaptive finite element methods for OCPs with control constraints.

\par \noindent

In recent years,  numerical analysis of OCPs with state constraints has been an active area of research. The articles \cite{Casas:1993:StateConst,casas:finite,casas:M,casas:14,DH:2007:StateCtrl,MRT:2005:Regular,Meyer:2008:StateConst} are devoted to the control problems with pointwise  state constraints. These articles are concentrated on the existence, uniqueness, regularity results of the optimal variables and also analyze asymptotic convergence of the errors in optimal variables.   The authors of \cite{RW:2012:StateCtrl_Apost} have considered the elliptic optimal control problem with state and control constraints, and derived reliable a posteriori error estimator. In \cite{hoppe:2010},  authors have used mixed-control state constraints as a relaxation of originally state constrained to avoid the intrinsic difficulties arising from measure-valued Lagrange multipliers in the case of pure state constraints OCP and derived residual type a posteriori error estimates.

  \par \noindent
A priori analysis of OCPs with integral state constraint is discussed in \cite{LYYG:2010:IntState,ZY:2015:LGS,lixin17, PS:2021:APNUM}. The authors of \cite{yuan2009} have derived a posteriori error estimates for a state-constrained OCP with integral state constraint. Recently in  \cite{chenz2019}, Chen {\it et al.} considered hp spectral element method  for integral state constrained elliptic optimal control problem and derived a posteriori error estimates for the coupled state and control approximation. The authors of \cite{ZY:2015:LGS} have considered Galerkin spectral approximation for an OCPs with state integral constraint in one dimension and derived  a priori and a posteriori error estimates. 
\par \noindent
In this article, we use a different approach to analyze adaptive finite element method for OCPs with integral state constraints. This approach avoids the use of the first order optimality conditions and therein the state and control constrained OCP can be reformulated into purely state constrained optimization problem. 
The optimal state is then obtained by solving a fourth order variational inequality. The main intent of this article is to derive a reliable and efficient a posteriori error estimator of a non-conforming finite element method for the elliptic distributed optimal control problem with two different set of constraints, namely (i) with integral state constraint and integral control constraint (ii) with integral state constraint and pointwise control constraints. 
We refer to \cite{BSZ:2013:C0IP,brenner:sung:2017,LYG:2009:StateConst,BGKS:2018:SC}, for the numerical analysis of constrained OCP based on the approach of reduction to purely state constrained optimization problem.
Recently in \cite{brenner:et:al:2019}, the authors have used $C^0$ interior penalty method for an elliptic state-constrained optimal control problem with Neumann boundary conditions and derived a priori and a posteriori error estimates.\\

Let $\Omega \subset\mathbb{R}^2$ be a convex polygonal domain with smooth boundary $\partial\Omega$. For any $1\leq p \leq \infty $ and $D \subset \O$, we denote $L^{p}(D)$ norm by $\|\cdot \|_{L^{p}(D)}$. We adopt the standard notations  $W^{m,p}(\Omega)$ and $W_{0}^{m,p}(\Omega)$ for Sobolev spaces for $p\in[1,\infty]$ and $m\geq 0$ equipped with norm $\|\cdot\|_{W^{m,p}(\Omega)}$ and seminorm $|\cdot|_{W^{m,p}(\Omega)}$. When $ p=2$, we denote $W^{m,p}(\Omega)$ by $H^m(\Omega)$ and $W_0^{m,p}(\Omega)$ by $H^m_0(\Omega)$ and corresponding norm and seminorm are denoted by $\|\cdot\|_{H^{m}(\Omega)}$ and  $|\cdot|_{H^{m}(\Omega)}$, respectively.  
We consider the following state and control constrained optimal control problem: to find $(\tilde{u},\tilde{y}) \in L^2(\Omega) \times H^1_0(\Omega)$ such that
\begin{eqnarray}
J(\tilde{u},\tilde{y})= \min_{(u,y) \in L^2(\Omega) \times H^1_0(\Omega)} \,J(u,y)\label{intro:functional}
\end{eqnarray}
 subject to  
\begin{eqnarray}
\begin{cases}
\int_{\Omega}\nabla y\cdot\nabla w\,dx=\int_{\Omega}uw\,dx\;\;\forall w\in H^1_0(\Omega)\;\;\\
\int_{\Omega} u \,dx\geq \delta_1,\;\;\;\;\;\\ 
\int_{\Omega}y\,dx \geq \delta_2, 
\end{cases}\label{intro:state:cons}
\end{eqnarray}
where $J(u,y)= \|y-y_d\|_{L^2(\Omega)}^2 +\frac{\beta}{2}\|u\|_{L^2(\Omega)}^2 $ with $y_d$ as the given desired state, $\delta_1, \delta_2 \in \mathbb{R}$ and $\beta>0$ is a given constant.  

 \par \noindent
 The rest of the article is organized as follows. In Section \ref{sec:2}, we obtain the characterization of the solution of the optimization problem \eqref{intro:functional}-\eqref{intro:state:cons} by the solution of a fourth order variational inequality  and discuss the optimality conditions of the underlying OCP with integral state as well as integral control constraints.  In Section \ref{sec:3}, we introduce notations and preliminary results required in the subsequent sections. Therein, we also discuss the finite element discretization of the continuous problem by a bubble enriched Morley finite element method and present the optimality conditions associated to the discrete problem. A posteriori error estimator of the underlying finite element method in introduced in Section \ref{sec:Apos}, followed by that reliability and efficiency estimates are established. In Section \ref{sec:OCPs}, we discuss a posteriori error bounds for an OCP with integral state and pointwise control constraints using the proposed finite element method. Finally, in Section \ref{sec:NumTests}, we present numerical results to illustrate the performance of derived a posteriori error estimators.
 
 \section{Continuous Variational Inequality and Optimality Conditions}\label{sec:2}
 This section is devoted to characterize the solution of (\ref{intro:functional})-(\ref{intro:state:cons}) by the solution of a variational inequality and discuss the associated optimality conditions.
\par \noindent
 For $u\in L^2(\Omega)$, Lax-Milgram lemma \cite{Ciarlet:1978:FEM} ensures the existence of a unique solution $y\in H^1_0(\Omega)$ satisfying the variational formulation 
    \begin{eqnarray}
  \int_{\Omega}\nabla y\cdot\nabla w\,dx=\int_{\Omega} u w\,dx ~~\qquad \forall w \in H^1_0(\Omega).
  \end{eqnarray}
 Moreover, from elliptic regularity  theory (cf. \cite{adams,grisvard}) we obtain  $y\in H^2(\Omega)$. Set $W=H^2(\Omega)\cap H_0^1(\Omega)$. Using $u=-\Delta y$, we can rewrite the optimization problem (\ref{intro:functional})-(\ref{intro:state:cons}) as follows: to find $ {y}^{*} \in \mathcal{K}$ such that
 \begin{align}\label{eq:MP}
  {y}^{*}&=\displaystyle argmin_{y\in \mathcal{K}}\Big[ \frac{1}{2}\int_{\Omega}(y-y_d)^2\,dx+\frac{\beta}{2}\int_{\Omega}(-\Delta y)^2\,dx\Big]
 \end{align} 
 where $\mathcal{K}$ is defined by
 \begin{eqnarray}
\mathcal{K} =\{w\in W:\,\int_{\Omega}w\,dx\geq \delta_2 \;\;\text{and}\;\; \int_{\Omega}(-\Delta w)~dx\geq\delta_1\;\}.\label{intro:k}  
 \end{eqnarray}
 The minimizer of \eqref{eq:MP} can further be characterized by the minimizer of the following optimization problem: find $ {y}^{*} \in \mathcal{K}$ such that
 \begin{align}\label{eq:MP1}
 {y}^{*}&=\displaystyle argmin_{y\in \mathcal{K}}\Big[\frac{1}{2}\mathcal{A}(y,y)-(y_d,y)\Big]
 \end{align} 
 where
 \begin{eqnarray}
 \mathcal{A}(v,w)=\beta\int_{\Omega}D^2v:D^2w\,dx+\int_{\Omega}vw\,dx,\;\;v,w\in W,\label{intro:state:modi}
 \end{eqnarray}
 with $ D^2v:D^2w=\displaystyle\sum_{i,j=1}^2\Big(\frac{\partial^2v}{\partial x_i\partial x_j}\Big)\Big(\frac{\partial ^2w}{\partial x_i\partial x_j}\Big)$ where $D^2$ denotes the Hessian matrix.
 \par \noindent
 We assume the following Slater condition holds \cite{fredi10}: there exists $y\in W$ satisfying $\int_{\Omega}y\,dx>\delta_2$ and $\int_{\Omega}(-\Delta y)\,dx\geq \delta_1$. This ensures that the set $\mathcal{K}$ is nonempty, together with closed and convex. Since the bilinear form $\mathcal{A}(\cdot,\cdot)$ is bounded, coercive and symmetric on $W$, by the standard theory (cf. \cite{Glowinsiki:1984:book,kinderlehrer}) there exists a unique solution $y^*\in \mathcal{K}$ of (\ref{eq:MP1}) satisfying 
 the following variational inequality
\begin{equation}
   \mathcal{A}({y}^*,w-{y}^*) \geq \int_{\Omega}y_d(w-{y}^*)\,dx\;\;\;\;\forall w\in \mathcal{K}.\label{variational:ineq}
  \end{equation}
 \par \noindent
Using the Lagrange multiplier approach, we obtain the following Karush-Kuhn-Tucker conditions  (cf. \cite{ito,luenberger}) together with complementarity conditions \eqref{lambda}-\eqref{mu_N}: there exist $\lambda\in \mathbb{R}$ and ${\mu}\in\mathbb{R}$  such that 
 \begin{eqnarray}
 \mathcal{A}(y^*,w)= \int_{\Omega}y_d w\,dx -\int_{\Omega}\lambda (\Delta w)\,dx+\int_{\Omega} \mu w\,dx\;\;\;\;\forall w\in W,\label{kkt}
 \end{eqnarray} 
 with 
 \begin{eqnarray}
\lambda \geq 0,\;\;&&\;\;\text{if}~~\;\int_{\Omega}(-\Delta {y}^*)\,dx=\delta_1,\label{lambda}\\
\lambda =0,\;\;&&\;\;\text{if}~~\;\int_{\Omega}(-\Delta {y}^{*})\,dx>\delta_1,\\
\mu \geq 0,\;\;&&\;\;\text{if}~~\;\int_{\Omega} {y}^{*}\,dx=\delta_2,\label{mu}\\
\mu=0,\;\;&&\;\;\text{if}~~\;\int_{\Omega} {y}^{*}\,dx>\delta_2\label{mu_N}.
\end{eqnarray} 
Note that, the adjoint state $p\in H^1_0(\Omega)$  satisfy
\begin{eqnarray}
\int_{\Omega}\nabla p\cdot \nabla w\,dx=\int_{\Omega}(y^*-y_d)\,w\,dx-\int_{\Omega}\mu w\,dx\;\;\forall w\in H^1_0(\Omega).\label{adjont:ct}
\end{eqnarray}
\noindent

\section{Notations and Finite Element Discretization}\label{sec:3}
In this section, we introduce the discrete control problem and present some useful tools required for subsequent analysis.
Let $\mathcal{T}_h$ be a regular triangulation of the domain $\Omega$. The following notations will be used throughout this article.
\begin{align*}
&\mathcal{T}_e: \text{set of  elements in}\;\mathcal{T}_h\;\text{that share the common edge} \; e,\\ 
&h_T:\, \text{the diameter of the triangle}\;\; T,\;\;\;\;
h=\max_{T\in\mathcal{T}_h} h_T\\
&\mathcal{V}_h: \; \text{set of all vertices of} \,\mathcal{T}_h,\\ 
&\mathcal{V}_T: \,\text{set of three vertices of} ~\,T,\\
&\mathcal{E}_h=\mathcal{E}_h^i\cup \mathcal{E}_h^b:\;\text{the set of the edges of the triangle in}\;\mathcal{T}_h, \text{where}\, \mathcal{E}_h^i(\text{resp.,}\, \mathcal{E}_h^b) \;\;\text{is the subset of}\, \mathcal{E}_h\, \\ &\;\;\;\;\;\;\;\;\;\text{consisting of edges interior to} \,\Omega\,(\text{resp., along}\, \partial\Omega),\\
&h_e:\,\text{length of an edge}\, e\in\mathcal{E}_h\\
&\Delta_h:\,\text{piesewise (element-wise) Laplacian operator}\\
&\mathbb{P}_k(T):~ \text{space of polynomials defined on T of degree less
than or equal to} ~k, ~k \geq 0~ \text{integer},\\
&X \lesssim Y: \text{there exists a positive constant} ~C~ \text{(independent of mesh parameter) such that} ~X \leq C Y,\\
&X \approx Y:  \text{there exists positive constants} ~C_1~\text{and}~ C_2 ~\text{such that}~C_1 Y \leq X \leq C_2 Y.
\end{align*}
Throughout this article, the constant $C$ will denote a positive generic constant.
\par \noindent
We denote by $H^k(\Omega,\mathcal{T}_h)$ the broken Sobolev space
\begin{eqnarray*}
H^k(\Omega,\mathcal{T}_h):=\{w\in L^2(\Omega): \,w_{T}=w|_{T}\in H^k(T)\;\;\;\;\forall T\in\mathcal{T}_h\}.
\end{eqnarray*}
Let $e\in \mathcal{E}_h^i$ be the common side shared by elements $T_{+}$ and $T_{-}$.
Further, suppose $n_{+}$ is the unit normal of $e$ pointing from $T_{+}$ to $T_{-}$, and $n_{-}=-n_{+}$. For any scalar valued function $w\in H^2(\Omega,\mathcal{T}_h)$, we define the jumps $\jump{\cdot}$, and averages $\mean{\cdot}$ across the edge $e$ as follows:
\begin{eqnarray*}
\jump{\frac{\partial w}{\partial n} }
=\frac{\partial w_{+}}{\partial n}\Big|_{e}-\frac{\partial w_{-}}{\partial n}\Big|_{e}\;\;\;\text{and}\;\;\;
\mean{\frac{\partial w}{\partial n}}
=\frac{1}{2}\Big(\frac{\partial w_{+}}{\partial n}\Big|_{e}+\frac{\partial w_{-}}{\partial n}\Big|_{e}\Big).
\end{eqnarray*}
For any $w\in H^3(\Omega,\mathcal{T}_h),$ we define
\begin{eqnarray*}
\jump{\frac{\partial^2 w}{\partial n^2} }
=\frac{\partial^2 w_{+}}{\partial n^2}\Big|_{e}-\frac{\partial^2 w_{-}}{\partial n^2}\Big|_{e}\;\;\;\text{and}\;\;\;
\mean{\frac{\partial^2 w}{\partial n^2}}
=\frac{1}{2}\Big(\frac{\partial^2 w_{+}}{\partial n^2}\Big|_{e}+\frac{\partial^2 w_{-}}{\partial n^2}\Big|_{e}\Big),
\end{eqnarray*}
where $w_{\pm}=w|_{T_{\pm}}$. For $e\in \mathcal{E}_h^b$, we choose $n$ be the unit outward normal of $e$ and let $T \in \mathcal{T}_h$ be such that $e=\partial T \cap \partial \Omega$. Set
\begin{eqnarray*}
\jump{\frac{\partial w}{\partial n} }
=\frac{\partial w|_{T}}{\partial n}\Big|_{e}\;\text{for any}\;w\in H^2(\Omega,\mathcal{T}_h).
\end{eqnarray*} 

\noindent
Before introducing the finite element spaces, we define for each triangle $T\in \mathcal{T}_h$ a cubic  bubble function $b_T \in \mathbb{P}_3(T)$ by
\begin{eqnarray*}
b_T= 60\lambda_1^T \lambda_1^T \lambda_3^T,
\end{eqnarray*}
where  $\lambda_i^T, ~i=1, 2, 3$ are the barycentric coordinates of $T$ associated with the vertices $p_i \in \cV_T$. \\
\par \noindent
\textbf{Discrete Spaces}:
Let $V_M$  denote the Morley finite element space \cite{Morley:1968} defined by
\begin{align*}
V_M=&\{w_h\in L^2(\Omega):\,w_h|_{T}\in \mathbb{P}_2(T) ~\forall~ T \in \mathcal{T}_h,\; w_h\,\text{is continuous at the vertices of}~ \mathcal{T}_h,~ \\&\quad \text{and the normal derivative of} \,w_h\;\text{is continuous  at the midpoint of the edges of}~ \mathcal{T}_h,~ \\ & \quad\text{and}  \,w_h\, \text{vanish on} \,\partial\Omega\},  
\end{align*}
and define the space $V_h$ as
\begin{align*}
V_h= \{ w_h \in L^2(\O): w_h|_T \in~ \text{span}({b_T})~ \forall T \in \cT_h \}.
\end{align*}
The finite element space $W_h$ is defined as 
$$W_h=V_M\oplus V_h.$$
 {{
 The  discrete norm $\|\cdot\|_h$ on $W_h$ is defined by 
 \begin{eqnarray*}
 \|w_h\|_h^2=\beta\sum_{T\in \mathcal{T}_h}|w_h|_{H^2(T)}^2+\|w_h\|_{L^2(\Omega)}^2 \quad \mbox{for} ~~ w_h\in W_h.\label{discrete:norm}
 \end{eqnarray*}
 }}
The discrete approximation of the convex set $\mathcal{K}$ is then given by
 \begin{align}\label{kh}
 \mathcal{K}_h=\{w_h\in W_h:\;\;\int_{\Omega}w_h\,dx\geq \delta_2\;\;\text{and}\;\;\int_{\Omega}(-\Delta_h w_h)\,dx\geq \delta_1\}.
 \end{align}
 
\par \noindent 
Next, we define the projection, interpolation and enriching operators and tabulate their approximation properties required in further analysis.\\
\par\noindent
\textbf{Discrete Operators}:
For any $T\in\mathcal{T}_h$ and $w\in L^1(T)$, define 
\begin{eqnarray}
Q_T(w)=\frac{1}{|T|}\int_{T}w\,dx.
\end{eqnarray}
Let $W_{pc,h}:=\{w\in L^1(\Omega):\, w|_{T}\in \mathbb{P}_0(T)\;\forall ~T\in\mathcal{T}_h\}$. Then, $Q_h:L^1(\Omega)\rightarrow W_{pc,h}$ is defined by setting $Q_h(w)|_{T}=Q_T(w)\;\;\text{for all}\;w\in L^1(\Omega)$, $T\in \cT_h$.\\
\par \noindent
Define interpolation operator $I_h:W\rightarrow W_h$  as: for $\xi \in W$,
\begin{align}
(I_h \xi)(p)&=\xi(p)\;\;\;\;\forall p\in\mathcal{V}_h,\label{interpo:1}\\
\int_{e}\frac{\partial(I_h \xi)}{\partial n}\,ds&=\int_{e}\frac{\partial\xi}{\partial n}\,ds,\;\;\;\;\forall e\in \mathcal{E}_h,\label{interpo:2}\\
Q_T(I_h \xi)&=Q_T(\xi),\;\;\;\;\forall T\in \mathcal{T}_h\label{interpo:3}.
\end{align}
The interpolation operator is well-defined and $(I_h w )|_T=w$ for any $w \in \mathbb{P}_2(T)$.
%

 \noindent
 For any $\xi \in W$, using (\ref{interpo:3}) we find
 \begin{equation}\label{eq:IntC1}
 \int_\O I_h \xi \,dx = \int_\O \xi \,dx.
 \end{equation}
 Further, a use of integration by parts and \eqref{interpo:2} yields
 \begin{eqnarray}
 \int_{T}\Delta(I_h\xi)\,dx=\sum_{e\in \partial T}\int_{e}\frac{\partial (I_h\xi)}{\partial n}\,ds=\int_{T}(\Delta \xi)\,dx\;\;\;\;\;\forall T\in \mathcal{T}_h,
 \end{eqnarray}
  which implies that
 \begin{eqnarray}
Q_h( \Delta_h(I_h\xi))=Q_h(\Delta \xi),\;\;\;\;\;\forall \xi\in W.\label{inter:err:3}
 \end{eqnarray}
 In view of (\ref{intro:k}), (\ref{kh}), \eqref{eq:IntC1} and (\ref{inter:err:3}), we have
 \begin{eqnarray}
 I_h \mathcal{K}\subset \mathcal{K}_h.\label{k:kh}
 \end{eqnarray}
 This relation also depicts that the discrete set $\mathcal{K}_h$ is non-empty. We would like to remark here that enriching the Morley finite element space $V_M$ by  the bubble function space $V_h$ plays a crucial role in obtaining \eqref{k:kh}.
 \par \noindent
Below, we state the stability and approximation properties of $I_h$, whose proof follows by using  Bramble Hilbert lemma  and scaling arguments; see  \cite{brennerscott,Ciarlet:1978:FEM} for details.
\begin{lemma}\label{lem:approx_proj}
Let $T\in \mathcal{T}_h$ and $s$ be an integer such that $ 0\leq s\leq 2$ and  $\psi\in H^{s}(T)$.  Then,
\begin{align}
|I_h \psi|_{H^{s}(T)} & \lesssim  | \psi|_{H^{s}(T)},  \qquad  0\leq s\leq 2 &&\label{stabPik0} \\
\sum_{k=0}^{s} h_{T}^{k-s} |\psi-I_h\psi|_{H^{k}(T)} & \lesssim | \psi|_{H^{s}(T)}, \qquad  0\leq s\leq 2. &&\label{stabPik2}
\end{align}

\end{lemma}
 
 \noindent
 Now we define an important tool for the analysis,  the enriching operator $E_h: W_h\rightarrow (M_h \oplus V_h) \cap W$, where $M_h$ is the Hsieh-Clough-Tocher macro element space \cite{Ciarlet:1978:FEM} associated with $\mathcal{T}_h$. The operator $E_h$ can be constructed by averaging techniques (cf. \cite{BWZ:2004:DPI,BGKS:2018:SC,PS:2021:APNUM}) satisfying
 \begin{eqnarray}
 \int_{e}\frac{\partial (E_hw_h)}{\partial n}\,ds=\int_{e}\frac{\partial w_h}{\partial n}\,ds\;\;\;\;\;\forall e\in \mathcal{E}_h,\label{enriching:2}\\
 \text{and}~~~
 \int_T E_h w_h \,dx = \int_T w_h \,dx\;\;\;\;\;\forall T\in \mathcal{T}_h \label{eq:EnrichP1}.
 \end{eqnarray}
  An application of integration by parts and (\ref{enriching:2}) leads to
 \begin{eqnarray}
 Q_h(\Delta E_h w_h)=Q_h(\Delta _h w_h).\label{eq:EnrichPP}
 \end{eqnarray}
Moreover, the enriching operator satisfies the following approximation properties (cf. \cite{BGKS:2018:SC}).

\begin{lemma}\label{Eh:lemma}
For any $w_h\in W_h$, we have 
\begin{eqnarray*}
&&\sum_{T\in\mathcal{T}_h}\Big(h_T^{-4}\|w_h-E_h w_h\|_{L^2(T)}^2+h_T^{-2}|w_h-E_h w_h|_{H^1(T)}^2+|w_h-E_h w_h|^2_{H^2(T)}\Big)\lesssim \|w_h\|_h^2,\\
&&\sum_{T\in\mathcal{T}_h}|w_h-E_hw_h|_{H^2(T)}^2\lesssim \sum_{e\in\mathcal{E}_h}\frac{1}{h_e}\Big\|\jump{\frac{\partial w_h}{\partial n}}\Big\|_{L^2(e)}^2\;\;\;\;\forall w_h\in W_h.
\end{eqnarray*}
\end{lemma}
\noindent We recall the following inverse and trace inequalities which will be useful in  later analysis \cite{Ciarlet:1978:FEM}.\\
\par \noindent
{\it Inverse Inequalities:} For any $w_h\in W_h$ and $1 \leq p\,, q  < \infty$,
\begin{align}
\|w_h\|_{W^{m,p}(T)} &  \lesssim h_{T}^{\ell -m}   h_{T}^{2\left(\frac{1}{p}-\frac{1}{q}\right)}
\|w_h\|_{W^{\ell,q}(T)}\quad &\forall\,T\in \cT_h,\label{eq:inverse}\\
\|\nabla w_h\|_{L^{p}(T)} &  \lesssim h_T^{-1}
\|w_h\|_{L^{p}(T)}\quad&\forall\,T\in \cT_h.\label{eq:inverse1}
\end{align}
\par \noindent
{\it Discrete trace inequality:}\label{{lem:trace}} Let $\psi\in W^{1,p}(T), ~T\in \mathcal{T}_h$
and let $e\in\mathcal{E}_h$ be an edge of $T$. Then for any $1 \leq p < \infty$, it holds that
{
\begin{equation}\label{eq:traceAg}
\|\psi\|_{L^{p}(e)}^p \lesssim h_e^{-1} \big( \|\psi\|_{L^{p}(T)}^p+
h_e^{p}\|\nabla \psi\|_{L^{p}(T)}^p\big).
\end{equation}
}


  \par\noindent
\textbf{Discrete Problem}:  The discrete form of the minimization problem (\ref{eq:MP1}) is defined as follows: Find $y_h^{*}\in \mathcal{K}_h$ such that
 \begin{eqnarray}
  y_h^{*}=argmin_{y_h\in\mathcal{K}_h}\Bigg(\frac{1}{2}\mathcal{A}_h(y_h,y_h)-(y_d,y_h) \Bigg),\label{Dis}
 \end{eqnarray}
 where 
 \begin{align}
 \mathcal{A}_h(w_h,v_h)=\beta\sum_{T\in\mathcal{T}_h}\int_{T}D^2 w_h:D^2v_h\,dx+\int_{\Omega}w_hv_h\,dx,~~\;\;\;\; w_h, v_h\in W_h.\label{def:discrete:norm}
 \end{align}

 \noindent
 Since $\mathcal{K}_h$ is non-empty, closed, convex and the  bilinear form $\mathcal{A}_h(\cdot,\cdot)$ is symmetric and positive definite on $W_h$, the discrete problem (\ref{Dis}) is well-posed and it's solution is characterized by the solution of the discrete variational inequality
\begin{eqnarray}
\mathcal{A}_h(y_h^{*},w_h-y_h^{*})\geq (y_d,w_h-y_h^*)\;\;\;\;\forall w_h\in\mathcal{K}_h.\label{discrete:variational:inequality}
\end{eqnarray} 
As in the case of the continuous problem, we have the following optimality conditions associated with the discrete problem \cite{fredi10}:
\begin{lemma}\label{lemma:KKT}
Let ${y}_h^{*} \in \mathcal{K}_h$  be the optimal solution of the discrete problem, then there exists Lagrange multipliers $\lambda_h \in \mathbb{R}$ and $\mu_h \in \mathbb{R}$ such that the following conditions hold:
\begin{align}
\mathcal{A}_h({y}_h^*, w_h)-\int_{\Omega}y_d w_h\,dx= \int_{\Omega} \mu_h w_h\,dx-\int_{\Omega}\lambda_h (\Delta_h w_h)\,dx, \quad \forall w_h\in W_h,\label{Dkkt}
 \end{align} 
 together with
 \begin{eqnarray}
 \mu_h\geq 0,\;\;\;\;\lambda_h\geq 0,\\
  \mu_h\Big(\delta_2-\int_{\Omega}{y}_h^{*}\,dx\Big)=0,\\
 \lambda_h\Big(\delta_1+\int_{\Omega}\Delta_h{y}_h^{*}\,dx\Big)=0.
 \end{eqnarray}
 \end{lemma}

\section{A Posteriori Error Analysis} \label{sec:Apos}
In this section we introduce a posteriori error estimator and present the first main result of the paper, namely,  the reliability analysis of the error estimator. Followed by that, we also discuss the efficiency estimates of a posteriori error estimator.
The contributions of error estimator are defined by
\begin{align*}
\eta_1^2&=\beta^{-1}\sum_{T \in \cT_h} h_T^4  \|y_d+\mu_h-{y}_h^*\|_{L^2(T)}^2 ,\\
\eta_2^2&=\beta\sum_{e\in\cE_h^i}\frac{1}{h_e}\Big\|\jump{\frac{\partial {y}^*_h}{\partial n} } \Big\|^2_{L^2(e)},\\
\eta_3^2 &=\beta\sum_{e \in \cE_h^i} h_e\Bigg\|\jump{\frac{\partial^2 {y}^*_h}{\partial n^2} } \Bigg\|_{L^2(e)}^2,\\
\eta_4^2&=\beta\sum_{e\in\cE_h^i}h_e^3\Bigg\|\jump{\frac{\partial (\Delta {y}_h^*)}{\partial n} } \Bigg\|^2_{L^2(e)},\\
\eta_5^2 &= \beta^{-1}\sum_{T \in \cT_h} h_T^2 |\lambda_h|^2.
\end{align*}
The full error estimator $\eta_h$ is given by
\begin{align}
\eta_h^2&= \eta_1^2 +\eta_2^2+\eta_3^2+\eta_4^2+\eta_5^2.\label{etah}
\end{align}
\subsection{Reliability of Error Estimator}
Below, we establish the reliability estimates of a posteriori error estimator $\eta_h$.
\begin{theorem}\label{thm:rel}
 Let ${y}^*$ and ${y}^*_h$ be solutions of variational inequalities \eqref{variational:ineq} and \eqref{discrete:variational:inequality}, respectively. Then, it holds that,
\begin{align*}
\|{y}^*-{y}^*_h\|_h\lesssim \eta_h.
\end{align*}
\end{theorem}
\begin{proof} 
We set $\phi={y}^*- E_h{y}_h^* \in W$ and let $\phi_h\in W_h$. Using the coercive property of the bilinear form $\mathcal{A}(\cdot,\cdot)$, (\ref{kkt}) and (\ref{Dkkt}) we obtain
\begin{align} \label{eq:rel1}
\|{y}^*- E_h{y}^*_h \|_h^2 &\lesssim \mathcal{A}({y}^*- E_h{y}_h^*, \phi) 
= \mathcal{A}({y}^*, \phi) -\mathcal{A}_h({y}_h^*, \phi) +\mathcal{A}_h({y}_h^*- E_h{y}_h^*, \phi) \nonumber\\
&\lesssim\int_{\Omega}y_d \phi\,dx -\int_{\Omega}\lambda (\Delta \phi)\,dx+\int_{\Omega} \mu \phi\,dx-\mathcal{A}_h(\bar{y}_h, \phi) +\mathcal{A}_h({y}_h^*- E_h{y}_h^*, \phi) \nonumber\\
&\lesssim \int_{\Omega}y_d (\phi-\phi_h)\,dx -\int_{\Omega}\lambda (\Delta \phi)\,dx+\int_{\Omega} \mu \phi\,dx+ \int_{\Omega}\lambda_h (\Delta_h \phi_h)\,dx-\int_{\Omega} \mu_h \phi_h\,dx\nonumber\\
&~~~~-\mathcal{A}_h({y}_h^*, \phi-\phi_h) +\mathcal{A}_h({y}_h^*- E_h{y}_h^*, \phi) \nonumber\\
&\lesssim \int_{\Omega}y_d (\phi-\phi_h)\,dx-\int_{\Omega} \mu_h (\phi_h-\phi)\,dx
-\mathcal{A}_h({y}_h^*, \phi-\phi_h) -\int_{\Omega}(\lambda-\lambda_h) (\Delta \phi)\,dx\nonumber\\ &~~~~ +\int_{\Omega} (\mu-\mu_h) \phi\,dx+ \int_{\Omega}\lambda_h \Delta_h (\phi_h- \phi)\,dx+\mathcal{A}_h({y}_h^*- E_h{y}_h^*, \phi).
\end{align}
\par \noindent
Now, we bound the terms of the right hand side of the last estimate.  The estimation of first three terms is discussed towards the end.  We first handle the rest terms other than the first three terms.
For the fourth term in \eqref{eq:rel1}, using the discrete and continuous complementarity conditions we get
\begin{align}\label{eq:rel2}
\int_{\Omega}(\lambda-\lambda_h) (-\Delta \phi)\,dx &=\int_{\Omega}(\lambda-\lambda_h) \big(-\Delta( {y}^*- E_h{y}_h^*) \big)\,dx \notag \\
&=\int_{\Omega}\lambda(- \Delta{y}^*+\Delta E_h{y}_h^*)\,dx - \int_{\Omega}\lambda_h ( -\Delta{y}^*+\Delta E_h{y}_h^*)\,dx
\notag\\
&=\lambda \Big(\int_{\Omega} (-\Delta{y}^*)~dx -\delta_1 \Big) +\lambda \Big(\delta_1+\int_{\Omega} \Delta E_h{y}_h^* ~dx \Big) \notag\\
&\hspace{0.4cm}-\lambda_h \Big(\int_{\Omega} (-\Delta{y}^*)~dx -\delta_1 \Big)-\lambda_h \Big(\delta_1+\int_{\Omega} \Delta E_h{y}_h^* ~dx \Big) \notag\\
&\leq\lambda \Big(\delta_1+\int_{\Omega} \Delta_h{y}_h^* ~dx \Big) -\lambda_h \Big(\delta_1+\int_{\Omega} \Delta_h{y}_h^* ~dx \Big)
   \notag\\
& \leq 0,
\end{align}
\par \noindent
where in obtaining second last estimate we have used that $\int_{\Omega} \Delta E_h{y}_h^* \,dx=\int_{\Omega} \Delta_h{y}_h^* ~dx$, $\lambda_h \Big(\int_{\Omega} (-\Delta{y}^*)~dx -\delta_1 \Big) \geq 0$  and then $\lambda \Big(\delta_1+\int_{\Omega} \Delta_h{y}_h^* ~dx \Big) \leq 0$.\\
\par \noindent
Next, we handle the fifth term of right hand side of \eqref{eq:rel1}. A use of \eqref{mu}, \eqref{mu_N}, \eqref{eq:EnrichP1} together with $\mu_h \geq 0$, $ \int_{\Omega} {y}^*~dx \geq \delta_2$ and $\mu \Big(\delta_2-\int_{\Omega} {y}_h^*\,dx  \Big)  \leq 0$, yields
\begin{align}\label{eq:rel3}
\int_{\Omega} (\mu-\mu_h) \phi\,dx &= \int_{\Omega} \mu ( {y}^*- E_h{y}_h^*)\,dx- \int_{\Omega} \mu_h ( {y}^*- E_h{y}_h^*)\,dx \notag\\&= \mu \Big(\int_{\Omega} {y}^*\,dx -\delta_2 \Big) +\mu \Big(\delta_2-\int_{\Omega} E_h{y}_h^*\,dx  \Big) \notag\\
& \hspace{0.4cm}-\mu_h \Big(\int_{\Omega} {y}^*~dx- \delta_2\Big)- \mu_h \Big(\delta_2-\int_{\Omega} E_h{y}_h^*~dx\Big) \notag\\
&\leq  \mu \Big(\delta_2-\int_{\Omega} E_h{y}_h^*\,dx  \Big) - \mu_h \Big(\delta_2-\int_{\Omega} E_h{y}_h^*~dx\Big)   \notag\\
& \leq \mu \Big(\delta_2-\int_{\Omega} {y}_h^*\,dx  \Big) - \mu_h \Big(\delta_2-\int_{\Omega} {y}_h^*~dx\Big) \notag\\
& \leq 0.
\end{align}
\par \noindent
The estimate on the last two terms of \eqref{eq:rel1} can be realized with an application of Lemmas \ref{Eh:lemma} and \ref{lem:approx_proj} as,
\begin{align}\label{eq:rel4}
\mathcal{A}_h({y}_h^*- E_h{y}_h^*, \phi) &\leq \|{y}_h^*- E_h{y}_h^*\|_h \|\phi\|_h
 \leq \eta_2 \|\phi\|_h,
\end{align}
and
\begin{align}\label{eq:rel5}
\int_{\Omega}\lambda_h \Delta_h (\phi_h- \phi)\,dx &= \sum_{T \in \cT_h} |\lambda_h| \int_T |\Delta(\phi_h- \phi)|~dx \leq  \sum_{T \in \cT_h} h_T|\lambda_h|  \|\Delta(\phi_h- \phi)\|_{L^2(T)} \notag\\
& \leq \sum_{T \in \cT_h} h_T|\lambda_h| |\phi|_{H^2(T)} 
 \leq \eta_5 \|\phi\|_h.
\end{align}
\noindent
We now proceed to handle the first three terms of \eqref{eq:rel1}. Performing integration by parts twice yields
\begin{align}
\mathcal{A}_h({y}_h^*, \phi-\phi_h) &= \beta\sum_{T \in \cT_h}\int_{T} D^2 {y}_h^*:D^2 (\phi-\phi_h)\,dx+\int_{\Omega} {y}_h^* (\phi-\phi_h)\,dx \notag\\
&=\beta\sum_{T \in \cT_h}\int_{T} \Delta^2 {y}_h^* (\phi-\phi_h)\,dx+ \beta\sum_{T \in \cT_h} \int_{\partial T}  \Big(\frac{\partial^2 {y}^*_h}{\partial n^2}\Big)  \Big(\frac{\partial(\phi-\phi_h) }{\partial n}\Big) \,ds \notag\\&~~~~+\beta \sum_{T \in \cT_h} \int_{\partial T}  \Big(\frac{\partial^2 {y}_h^*}{\partial n \partial t} \Big) \Big(\frac{\partial(\phi-\phi_h) }{\partial t}\Big)\,ds-\beta\sum_{T \in \cT_h} \int_{\partial T} \Big(\frac{\partial \Delta {y}_h^*}{\partial n}\Big) (\phi-\phi_h)\,ds\notag \\&~~~~~+\int_{\Omega}{y}_h^* (\phi-\phi_h)\,dx \notag\\
&= \beta\sum_{e \in \cE_h^i} \int_{e}  \jump{\frac{\partial \Delta{y}_h^*}{\partial n}}(\phi-\phi_h)\,ds  -\beta\sum_{e \in \cE_h^i} \int_e \jump{\frac{\partial^2 {y}_h^*}{\partial n^2}} \mean{\frac{\partial(\phi-\phi_h) }{\partial n}} ~ds\notag\\&~~~~- \beta\sum_{e \in \cE_h}\int_e \mean{\frac{\partial^2 {y}_h^*}{\partial n^2}} \jump{\frac{\partial(\phi-\phi_h) }{\partial n}}\,ds-\beta\sum_{e \in \cE_h^i} \int_e \jump{\frac{\partial^2 {y}_h^*}{\partial n\partial t}} {\frac{\partial(\phi-\phi_h) }{\partial t}} ~ds\notag\\
&~~~~~+\int_{\Omega} {y}_h^* (\phi-\phi_h)\,dx.\label{star:1}
\end{align}
Thus,
\begin{align*}
&\int_{\Omega}y_d (\phi-\phi_h)\,dx-\int_{\Omega} \mu_h (\phi_h-\phi)\,dx -\mathcal{A}_h({y}_h^*, \phi-\phi_h) 
=\int_{\Omega}(y_d+\mu_h-{y}_h^*) (\phi-\phi_h)\,dx\\&~~~~~-\beta\sum_{e \in \cE_h^i} \int_{e}  \jump{\frac{\partial \Delta{y}_h^*}{\partial n}}(\phi-\phi_h)\,ds+\beta\sum_{e \in \cE_h^i} \int_e \jump{\frac{\partial^2 {y}_h^*}{\partial n^2}} \mean{\frac{\partial(\phi-\phi_h) }{\partial n}} ~ds\notag\\
&~~~~~+\beta\sum_{e \in \cE_h}\int_e \mean{\frac{\partial^2 {y}_h^*}{\partial n^2}} \jump{\frac{\partial(\phi-\phi_h) }{\partial n}}\,ds
+\beta\sum_{e \in \cE_h^i} \int_e \jump{\frac{\partial^2 {y}_h^*}{\partial n\partial t}} {\frac{\partial(\phi-\phi_h) }{\partial t}} ~ds.\notag
\end{align*}
Now, we estimate the terms of right hand ride as follows: a use of the Cauchy-Schwarz inequality and Lemma \ref{lem:approx_proj}  yields
\begin{eqnarray}
\Big|\int_{\Omega}(y_d+\mu_h-{y}_h^*) (\phi-\phi_h)\,dx\Big| &\leq& \sum_{T\in \cT_h} h_T^2\|y_d+\mu_h-{y}_h^*\|_{L^2(T)}  h_T^{-2}\|\phi-\phi_h\|_{L^2(T)}\nonumber \\
& \lesssim &\sum_{T\in \cT_h} h_T^2  \|y_d+\mu_h-{y}_h^*\|_{L^2(T)} |\phi|_{H^2(T)}\nonumber\\
& \lesssim& \eta_1 \|\phi \|_{h}.\label{estimate:1}
\end{eqnarray}
Using  Cauchy-Schwarz inequality, discrete trace inequality and Lemma \ref{lem:approx_proj}, we find
\begin{align}
\Big|\beta\sum_{e \in \cE_h^i} \int_e \jump{\frac{\partial^2 {y}_h^*}{\partial n^2}} &\mean{\frac{\partial(\phi-\phi_h) }{\partial n}} ~ds\Big| \leq \notag\\
& \beta\Bigg(\sum_{e \in \cE_h^i} h_e \Big\|\jump{\frac{\partial^2 {y}_h^*}{\partial n^2}} \Big\|_{L_2(e)}^2 \Bigg)^{1/2} \Bigg(\sum_{e \in \cE_h^i} h_e^{-1} \Bigg\|\mean{\frac{\partial(\phi-\phi_h) }{\partial n}} \Bigg\|_{L_2(e)}^2 \Bigg)^{1/2} \nonumber\\
& \lesssim \eta_3  \|\phi \|_{h}.\label{es:3}
\end{align}
Invoking Lemma \ref{lem:approx_proj} together with inverse inequality and discrete trace inequality \eqref{eq:traceAg}, we obtain
\begin{align}
\Big|\beta \sum_{e\in\cE_h^i}\int_e\jump{\frac{\partial^2 {y}_h^*}{\partial n\partial t}} &\frac{\partial (\phi-\phi_h)}{\partial t}\,ds\Big|\leq \notag\\
 &\beta\Bigg(\displaystyle\sum_{e\in \cE_h^i} h_e \Big\|\jump{\frac{\partial^2{y}_h^*}{\partial n\partial t}}\Big\|_{L^2(e)}^2\Bigg)^{\frac{1}{2}}\Bigg(\displaystyle\sum_{e\in\cE_h^i}h_e^{-1}\Bigg\|\frac{\partial(\phi-\phi_h)}{\partial t}\Bigg\|_{L^2(e)}^2\Bigg)^{\frac{1}{2}}\notag\\
&\lesssim \eta_2 \|\phi\|_{h}
\end{align}
and,
\begin{align}\label{eq:rel9}
\Big|\beta\sum_{e\in\cE_h^i}\int_{e}\jump{\frac{\partial (\Delta{y}_h^*)}{\partial n}}&(\phi-\phi_h)\,ds \Big|\leq  \notag \\ &\beta\Bigg(\sum_{e\in\cE_h^i}h_e^3 \Bigg\|\jump{\frac{\partial(\Delta {y}_h^*)}{\partial n}}\Bigg\|_{L^2(e)}^2\Bigg)^{\frac{1}{2}}\Big(\sum_{e\in\cE_h^i} h_e^{-3}\|\phi-\phi_h\|_{L^2(e)}^2\Big)^{\frac{1}{2}}\nonumber\\
&\lesssim  \eta_4 \|\phi\|_{h}.
\end{align}
\par \noindent
Finally, combining the estimates \eqref{eq:rel2}-\eqref{eq:rel9} together with \eqref{eq:rel1}, we get the desired result.
\end{proof}
\noindent
In order to obtain the reliability estimates for Lagrange multiplier errors $|\mu-\mu_h|\;\; \text{and} \;\;|\lambda-\lambda_h| $, we introduce the auxiliary variables $z_h\in W$ and $\hat{p}_h\in H^1_0(\Omega)$ satisfying the following equations.
\begin{eqnarray}
\mathcal{A}(z_h,w)=\int_{\Omega}y_d w\,dx+\int_{\Omega}\lambda_h\,(-\Delta w)\,dx+\int_{\Omega}\mu_h\,w\,dx\;\;\;\;\forall w\in W.\label{aux:1}
\end{eqnarray}
and
\begin{eqnarray}
\int_{\Omega}\nabla \hat{p}_h\cdot \nabla w\,dx=\int_{\Omega}(z_h-y_d)w\,dx-\int_{\Omega}\mu_h w\,dx\;\;\;\;\forall w\in H^1_0(\Omega).\label{adjoint:aux}
\end{eqnarray} 
The well-posedness of these auxiliary problems \ref{aux:1} and \ref{adjoint:aux} follows from Lax-Milgram lemma \cite{Ciarlet:1978:FEM}. These auxiliary problems help in estimating the errors in Lagrange multipliers. In the next lemma, we estimate the error $\|z_h-{y}^*\|_{h}$.
\begin{lemma}\label{thm:3}
There exists a positive constant $C$, depending only on the shape regularity of $\mathcal{T}_h$, such that
\begin{eqnarray*}
\|z_h-{y}^*\|_{h}\leq C\eta_h.
\end{eqnarray*}
\end{lemma}
\begin{proof}
\par \noindent
We have,
\begin{eqnarray}
\|z_h-{y}^*\|_{h}\lesssim \Big(\|z_h-{y}_h^*\|_h+\|{y}_h^*-{y}^*\|_h\Big).
\end{eqnarray}
The estimation of $\|z_h-{y}_h^*\|_h$ follows in similar steps as in Theorem \ref{thm:rel}. For completeness, we briefly discuss the proof.
A use of triangle inequality gives
\begin{eqnarray}
\|z_h-{y}_h^*\|_h\leq \|z_h-E_h{y}_h^*\|_h+\|E_h{y}_h^*-{y}_h^*\|_h.\label{aux:2}
\end{eqnarray}
For $\phi=z_h-E_h{y}^*_h\in W$ and $\phi_h=I_h \phi \in W_h$, a use of coercive property of $\mathcal{A}(\cdot,\cdot)$, (\ref{Dkkt}) and (\ref{aux:1}) leads to
\begin{eqnarray}\label{eq:Rel10}
\|z_h-E_h{y}_h^*\|^2_{h}&=&\|\phi\|^2_{h}\lesssim \mathcal{A}(z_h-E_h{y}_h^*,\phi)\nonumber\\
&\lesssim&\mathcal{A}(z_h,\phi)-\mathcal{A}_h({y}_h^*,\phi)+\mathcal{A}_h({y}_h^*-E_h{y}_h^*,\phi)\nonumber\\
&\lesssim&\int_{\Omega}y_d\phi\,dx+\int_{\Omega}\lambda_h(-\Delta \phi)\,dx+\int_{\Omega}\mu_h \phi\,dx-\mathcal{A}_h({y}_h^*,\phi)+\mathcal{A}_h({y}_h^*-E_h{y}_h^*,\phi)\nonumber\\
&\lesssim& \int_{\Omega}y_d(\phi-\phi_h)\,dx+\int_{\Omega}\lambda_h(-\Delta(\phi-\phi_h))\,dx+\int_{\Omega}\mu_h(\phi-\phi_h)\,dx\nonumber\\&&-\mathcal{A}_h({y}_h^*,\phi-\phi_h)+\mathcal{A}_h({y}_h^*-E_h{y}_h^*,\phi).
\end{eqnarray}
Each term of right hand side of the last equation are estimated as in Theorem \ref{thm:rel}, so we omit the details.
 The estimate of $\|z_h-{y}_h^*\|_h$ by $\eta_h$ is then realized by a use of \eqref{aux:2}, \eqref{eq:Rel10} and Lemma \ref{Eh:lemma}.
\end{proof}
\par \noindent
Next, we show that the error in Lagrange multipliers can be estimated in terms of $\|z_h-{y}^*\|_h$.
\begin{lemma}\label{lemma:adjoint}
There exists a positive constant $C$ depending only on the shape regularity of $\mathcal{T}_h$, such that
\begin{eqnarray}
|\mu-\mu_h|&\leq& C\|y^*-z_h\|_{L^2(\Omega)},\label{est:1}\\
|\lambda-\lambda_h|&\leq& C\|y^*-z_h\|_{h}.\label{est:2}
\end{eqnarray}
\end{lemma}
\begin{proof}
Upon subtracting (\ref{adjont:ct}) and (\ref{adjoint:aux}), we find
\begin{eqnarray}
\int_{\Omega}\nabla(p-\hat{p}_h)\cdot\nabla w\,dx=\int_{\Omega}(y^*-z_h)w\,dx-\int_{\Omega}(\mu-\mu_h)w\,dx \;\;\forall w\in H^1_0(\Omega).\label{adjoint:aux:2}
\end{eqnarray}
We choose the cut-off function $\psi\in\mathcal{C}_0^{\infty}(\Omega)$ with $\frac{1}{|\Omega|}\int_{\Omega}\psi\,dx=1$ and $\|\psi\|_{H^1(\Omega)}\leq C$. Let $\hat{C}=\frac{1}{|\Omega|}\int_{\Omega}(p-\hat{p}_h)\,dx$, we observe that $\hat{C}\psi\in \mathcal{C}_0^{\infty}(\Omega)\subset H^1_0(\Omega)$.  Take $w=p-\hat{p}_h-\hat{C}\psi\in H^1_0(\Omega)$ in (\ref{adjoint:aux:2}) to obtain
\begin{eqnarray*}
\int_{\Omega}\nabla (p-\hat{p}_h)\cdot\nabla (p-\hat{p}_h-\hat{C}\psi)\,dx=\int_{\Omega}(y^*-z_h)(p-\hat{p}_h-\hat{C}\psi)\,dx-\int_{\Omega}(\mu-\mu_h)(p-\hat{p}_h-\hat{C}\psi)\,dx.
\end{eqnarray*}
Using the fact $\mu-\mu_h \in \mathbb{R}$ and $\int_{\Omega}(p-\hat{p}_h-\hat{C}\psi)\,dx=0$, we obtain
\begin{eqnarray}
\|\nabla (p-\hat{p}_h)\|^2&=&\int_{\Omega}\nabla (p-\hat{p}_h)\hat{C}\psi\,dx+\int_{\Omega}(y^*-z_h)(p-\hat{p}_h-\hat{C}\psi)\,dx\nonumber\\
&\lesssim &\|\nabla (p-\hat{p}_h)\|_{L^2(\Omega)}\|\hat{C}\psi\|_{L^2(\Omega)}+\|y^*-z_h\|_{L^2(\Omega)}\|p-\hat{p}_h-\hat{C}\psi\|_{L^2(\Omega)}\nonumber\\
&\lesssim & |\hat{C}|\|\nabla (p-\hat{p}_h)\|_{L^2(\Omega)}\|\psi\|_{L^2(\Omega)}+\|y^*-z_h\|_{L^2(\Omega)}\{\|p-\hat{p}_h\|_{L^2(\Omega)}+|\hat{C}|\|\psi\|_{L^2(\Omega)}\}\nonumber\\
&\lesssim &\|p-\hat{p}_h\|_{L^2(\Omega)}\|\nabla (p-\hat{p}_h)\|_{L^2(\Omega)}\|\psi\|_{L^2(\Omega)}+\|y^*-z_h\|_{L^2(\Omega)}\{\|p-\hat{p}_h\|_{L^2(\Omega)}\nonumber\\&&~~+\|p-\hat{p}_h\|_{L^2(\Omega)}\|\psi\|_{L^2(\Omega)}\},\label{est:1:1}
\end{eqnarray}
where in obtaining the last estimate, we have used that $|\frac{1}{|\Omega|}\int_{\Omega}(p-\hat{p}_h)\,dx|\lesssim \|p-\hat{p}_h\|_{L^2(\Omega)}$. An application of Poincar\'{e} inequality  keeping in view the construction of $\psi$ and the standard kick back argument  leads to


{
\begin{eqnarray}\label{eq:IntP}
\|\nabla (p-\hat{p}_h)\|\lesssim \|y^*-z_h\|_{L^2(\Omega)}.
\end{eqnarray}
}

\par \noindent
In view of \eqref{adjoint:aux:2} and \eqref{eq:IntP}, we get the desired estimate (\ref{est:1}).  Upon  subtracting (\ref{aux:1}) from (\ref{kkt}), we get
\begin{eqnarray}\label{eq:IntP1}
\int_{\Omega}(\lambda-\lambda_h)(\Delta w)\,dx=\mathcal{A}(z_h-{y}^*,w)+\int_{\Omega}(\mu-\mu_h)w\,dx \;\;\;\forall w\in W.
\end{eqnarray}
Then, the estimate \eqref{est:2} can be realized from \eqref{eq:IntP1} and \eqref{est:1}.
%
\end{proof}
\par \noindent
Finally, in view of lemma \ref{thm:3} and \ref{lemma:adjoint}, we have the estimation of error in Lagrange multipliers by the error estimator $\eta_h$.

\subsection{Local Efficiency Estimates}
In this section, we derive the local efficiency estimates of a posteriori error estimator $\eta_h$ obtained in last subsection. Therein, we use the standard bubble function techniques \cite{verfurth:1995}. We have discussed the main ideas involved in proving these efficiency estimates and skipped the standard details.
For $z\in H^4(\Omega,\mathcal{T}_h)$ and $v\in H^2(\Omega,\mathcal{T}_h)$, we have the following integration by parts formula.
\begin{eqnarray}
\int_T(\Delta^2 z) v\,dx=\int_{T}D^2 z:D^2 v\,dx+\int_{\partial T}\frac{\partial \Delta z}{\partial n} v\,ds-\int_{\partial T}\frac{\partial^2 z}{\partial n \partial t}\frac{\partial v}{\partial t}\,ds-\int_{\partial T}\frac{\partial^2 z}{\partial n^2}\frac{\partial v}{\partial n}\,ds.
\end{eqnarray}
\par \noindent
Upon summing up for all $T\in\mathcal{T}_h$, we obtain
\begin{align}\label{eq:IntF}
\displaystyle\sum_{T\in\mathcal{T}_h}\int_{T}\Delta^2 z v\,dx&=\displaystyle\sum_{T\in\mathcal{T}_h}\int_{T}D^2z:D^2 v\,dx+\displaystyle\sum_{e\in\mathcal{E}_h^i\cup\mathcal{E}_h^b}\int_e\mean{\frac{\partial \Delta z}{\partial n}}[[v]]\, ds\notag \\&+\displaystyle\sum_{e\in\mathcal{E}_h^i}\int_{e}\jump{\frac{\partial (\Delta z)}{\partial n} }\{\{v\}\}\,ds+\sum_{e\in\mathcal{E}_h^i}\int_e\jump{\frac{\partial^2 z}{\partial n^2} }\mean{\frac{\partial v}{\partial n}}\,ds \notag\\
&+\displaystyle\sum_{e\in\mathcal{E}_h}\int_{e}\mean{\frac{\partial^2 z}{\partial n^2}}\jump{\frac{\partial v}{\partial n} }\,ds+\displaystyle\sum_{e\in\mathcal{E}_h^i}\int_{e}\jump{\frac{\partial^2 z}{\partial n\partial t} }\mean{\frac{\partial v}{\partial t} }\,ds\notag\\
&+\displaystyle\sum_{e\in\mathcal{E}_h}\int_{e}\mean{\frac{\partial^2 z}{\partial n\partial t} } \jump{\frac{\partial v}{\partial n} }\,ds.
\end{align}
\par \noindent
\begin{theorem}\label{thm:EffEstimates}
It holds that,
\begin{align}
h_T^2  \|y_d+\mu_h-{y}_h^*\|_{L^2(T)}&\lesssim \|{y}^*-{y}^*_h\|_{2,T}+h_T^2\|\mu-\mu_h\|_{L^2(T)}+\|\lambda-\lambda_h\|_{L^2(T)}+Osc(y_d;T)\;\;\;\;\;\forall T\in\mathcal{T}_h, \label{eq:eff1}\\
{\beta}{h_e^{-1/2}}\Big\|\jump{\frac{\partial {y}^*_h}{\partial n} } \Big\|_{L^2(e)}&\lesssim \sum_{T \in \mathcal{T}_e}\Big( \|{y}^*-{y}^*_h\|_{2,T}+h_T^2\|\mu-\mu_h\|_{L^2(T)}+\|\lambda-\lambda_h\|_{L^2(T)} \Big)\nonumber\\ & \qquad+Osc(y_d;\mathcal{T}_e)\;\;\;\;\forall e\in \mathcal{E}_h^i, \label{eq:eff2}\\
\beta h_e^{1/2}\Bigg\|\jump{\frac{\partial^2 {y}^*_h}{\partial n^2} } \Bigg\|_{L^2(e)}&\lesssim \sum_{T \in \mathcal{T}_e}\Big(\|{y}^*-{y}^*_h\|_{2,T}+h_T^2\|\mu-\mu_h\|_{L^2(T)}+\|\lambda-\lambda_h\|_{L^2(T)} \Big) \nonumber\\& \qquad +Osc(y_d;\mathcal{T}_e)\;\;\;\;\forall e\in \mathcal{E}_h^i,  \label{eq:eff3}\\
\beta{h_e^{3/2}\Bigg\|\jump{\frac{\partial (\Delta {y}_h^*)}{\partial n} } \Bigg\|_{L^2(e)}}&\lesssim \sum_{T \in \mathcal{T}_e}\Big(\|{y}^*-{y}^*_h\|_{2,T}+h_T^2\|\mu-\mu_h\|+\|\lambda-\lambda_h\|_{L^2(T)} \Big) \nonumber\\&\qquad+Osc(y_d;\mathcal{T}_e)\;\;\;\;\forall e\in \mathcal{E}_h^i,  \label{eq:eff4}\\
{h_T |\lambda_h|}&\lesssim \|{y}^*-{y}^*_h\|_{2,T}+h_T^2\|\mu-\mu_h\|_{L^2(T)}+\|\lambda-\lambda_h\|_{L^2(T)}+Osc(y_d;T)\;\;\;\;\;\forall T\in\mathcal{T}_h,\label{eq:eff5}
\end{align}
where $\|w\|_{2,T}:=\beta|w|_{H^2(T)}+h_T^2\|w\|_{L^2(T)}$ for any $w \in H^2(\Omega, \mathcal{T}_h)$, 
 $Osc(y_d;T):=h_T^2\|y_d-\bar{y}_d\|_{L^2(T)}^2$ with  $\bar{y}_d:=\frac{1}{|T|}\int_T y_d\,dx$ and $\mathcal{T}_e$ denotes the union of elements sharing the edge $e$.
\end{theorem}

\begin{proof}
\par\noindent
(i)\textbf{(Local bound for $\eta_1$)}
Let $\tilde{b}_T$ be a polynomial bubble function vanishing up to the first order on $\partial T$, i.e., $\tilde{b}_T$ and $\nabla \tilde{b}_T$ vanish on $\partial T$, and set $\phi_T=(\bar{y}_d+\mu_h-{y}_h^*) \tilde{b}_T$. 
Let $\tilde\phi$ be the extension of $\phi_T$ to $\bar\Omega$ by zero, clearly $\tilde\phi \in W$. We further have,
\begin{eqnarray}
\|\phi_T\|_{L^2(T)}\lesssim \|\bar{y}_d+\mu_h-{y}_h^*\|_{L^2(T)},
\end{eqnarray}
\par \noindent
and
\begin{eqnarray}\label{effi:1}
\|\bar{y}_d+\mu_h-{y}_h^*\|_{L^2(T)}^2&\approx&\int_{T}(\bar{y}_d+\mu_h-{y}_h^*)\phi_T\,dx\nonumber\\
&=&\int_{T}(y_d+\mu_h-{y}_h^*)\phi_T\,dx+\int_{T}(\bar{y}_d-y_d)\phi_T\,dx.
\end{eqnarray} 
\par \noindent
Using equation (\ref{kkt}) and the fact that,  ${\beta\int_{\Omega}D^2\bar{y}_h:D^2\tilde{\phi}\,dx=\int_{\Omega}\lambda_h(-\Delta_h\tilde{\phi})\,dx}$,
 we have 
 
\begin{align}\label{Eqn1}
\int_{T}(y_d+\mu_h-{y}_h^*)\phi_T\,dx&=\beta\int_{\Omega} D^2{y}^*:D^2\tilde{\phi}\,dx+\int_{\Omega}{y}^*\tilde{\phi}\,dx+\int_{\Omega}\lambda(\Delta \tilde{\phi})dx\nonumber\\& \quad -\int_{\Omega}\mu\tilde{\phi}\,dx+\int_{\Omega}\mu_h\tilde{\phi}\,dx-\int_{\Omega}{y}_h^*\tilde{\phi}\,dx\nonumber\\
&=\beta\int_{\Omega}D^2({y}^*-{y}_h^*):D^2\tilde{\phi}\,dx+\int_{\Omega}({y}^*-{y}_h^*)\tilde{\phi}\,dx\nonumber\\&\quad+\int_{\Omega}(\mu_h-\mu)\tilde{\phi}\,dx+\int_{\Omega}(\lambda-\lambda_h)\Delta_h\tilde{\phi}\,dx\nonumber\\
&\lesssim  \Big(\beta\,|{y}^*-{y}_h^*|_{H^2(T)}|\phi_T|_{H^2(T)}+\|{y}^*-{y}_h^*\|_{L^2(T)}\|\phi_T\|_{L^2(T)}\nonumber\\&~~~+\|\mu-\mu_h\|_{L^2(T)}\|\phi_T\|_{L^2(T)}+\|\lambda-\lambda_h\|_{L^2(T)}\|\Delta {\phi}_T\|_{L^2(T)}\Big)\nonumber\\
&\lesssim \Big(\beta h_T^{-2} \,|{y}^*-{y}_h^*|_{H^2(T)}+\|{y}^*-{y}_h^*\|_{L^2(T)}+\|\mu-\mu_h\|_{L^2(T)}\nonumber\\&\quad+h_T^{-2}\|\lambda-\lambda_h\|_{L^2(T)}\Big)\|\phi_T\|_{L^2(T)}.
\end{align}
where in the last step, we used the inverse estimate $|\phi_T|_{H^2(T)}\leq Ch_T^{-2}\|\phi_T\|_{L^2(T)}$.
Thus, from equation \eqref{effi:1} and \eqref{Eqn1}, we obtain
\begin{eqnarray*}
h_T^2\|\bar{y}_d+\mu_h-{y}_h^*\|_{L^2(T)}\lesssim \|{y}^*-{y}_h^*\|_{2,T}+h_T^2\|\mu-\mu_h\|_{L^2(T)}+\|\lambda-\lambda_h\|_{L^2(T)}+h_T^2\|y_d-\bar{y}_d\|_{L^2(T)},
\end{eqnarray*}
\noindent
thus we get the desired estimate.\\

\par \noindent
(ii)\textbf{(Local bound for $\eta_2$)}
{ We skip the proof of \eqref{eq:eff2} which follows using standard bubble function techniques together with the realization that}

\begin{eqnarray}
{h_e^{-1/2}}\Big\|\jump{\frac{\partial {y}^*_h}{\partial n} }\Big\|_{L^2(e)}= {h_e^{{-1}/{2}}}\Big\|\jump{\frac{\partial ({y}^*-{y}_h^*)}{\partial n}}\Big\|_{L^2(e)}\;\;\;\;\forall e\in\mathcal{E}_h^i.
\end{eqnarray}

\noindent
(iii)\textbf{(Local bound for $\eta_3$)} 
Let $\Theta=\beta\jump{\frac{\partial^2 {y}_h^*}{\partial n^2}}$ along $e$ and define ${\theta}_1\in \mathbb{P}_1(\mathcal{T}_e)$ by
\begin{eqnarray}
\theta_1=0\;\;\;\;\text{and}\;\;\;\;\frac{\partial \theta_1}{\partial n}=\Theta\;\;\text{on the edge}\; e.\label{edge:1}
\end{eqnarray}
It is easy to verify that
$\|\theta_1\|_{1,T_{\pm}}\approx h_e|\Theta|\;\;\;\;\text{and}\;\;\;\;|\theta_1|_{\infty,T_{\pm}}\approx h_e|\Theta|$.
Next, define $\theta_2\in \mathbb{P}_8(\mathcal{T}_e)$ satisfying the following properties:\\
(a) $\theta_2$ is positive on the edge $e$ and takes unit value at the midpoint of the edge.\\
(b) $\theta_2$ vanishes up to first order on $(\partial T_{+}\cup \partial T_{-})\setminus e$.\\
It follows from the scaling that

\begin{eqnarray}
{\|\theta_2\|_{1,T_{\pm}}\approx  1 \approx \|\theta_2\|_{\infty,T_{\pm}}.\label{edge:2}}
\end{eqnarray}

\par \noindent
Using integration by parts formula \eqref{eq:IntF}, Poincar\'{e} inequality, inverse inequality and equations (\ref{kkt}), (\ref{Dkkt}), we find 
\begin{eqnarray*}
\beta^2\Big\|\jump{\frac{\partial^2 {y}_h^*}{\partial n^2} }\Big\|^2_{L^2(e)}&=&\int_{e}\beta^2\jump{\frac{\partial^2 {y}_h^*}{\partial n^2} }^2\,ds=\int_e \beta\jump{\frac{\partial^2 {y}_h^*}{\partial n^2} }\Theta\,ds\\&\lesssim &\beta \int_{e}\jump{\frac{\partial^2 {y}_h^*}{\partial n^2} }\frac{\partial \theta_1}{\partial n}\theta_2\,ds
=\beta\int_{e}\jump{\frac{\partial^2 {y}_h^*}{\partial n^2} }\frac{\partial (\theta_1\theta_2)}{\partial n}\,ds\\&=&-\beta\sum_{T\in\mathcal{T}_e}\int_{T}D^2 y_h^*:D^2(\theta_1\theta_2)\,dx\\
&=&\beta\sum_{T\in\mathcal{T}_e}\int_{T}D^2({y}^*-{y}_h^*):D^2(\theta_1\theta_2)\,dx-\beta\sum_{T\in\mathcal{T}_e}\int_{T}D^2{y}^*:D^2(\theta_1\theta_2)\,dx\\
&\lesssim & \sum_{T\in\mathcal{T}_e}\Big\{\beta\int_{T}D^2({y}^*-{y}_h^*):D^2(\theta_1\theta_2)\,dx-\int_T(y_d+\mu_h-{y}_h^*)\,\theta_1\theta_2\,dx\nonumber\\
&&\qquad+\int_{T}(\mu_h-\mu)\theta_1\theta_2\,dx+\int_{T}({y}^*-{y}_h^*)\theta_1\theta_2\,dx+\int_{T}(\lambda-\lambda_h)\Delta_h (\theta_1\theta_2)\,dx\Big\}\nonumber\\
&\lesssim & \sum_{T\in\mathcal{T}_e}\Big\{\beta|{y}^*-{y}_h^*|_{H^2(T)}+h_T^2\|{y}^*-{y}_h^*\|_{L^2(T)}+h_T^2\|y_d+\mu_h-{y}_h^*\|_{L^2({T})} \nonumber\\&&\qquad +h_T^2\|\mu-\mu_h\|_{L^2(T)}+\|\lambda-\lambda_h\|_{L^2(T)} \Big\}|\theta_1\theta_2|_{H^2(\mathcal{T}_e)}\nonumber\\
&\lesssim & \sum_{T\in\mathcal{T}_e}\Big\{\|{y}^*-{y}_h^*\|_{2,{T}}+h_T^2\|y_d+\mu_h-{y}_h^*\|_{L^2({{T})}} \nonumber+h_T^2\|\mu-\mu_h\|_{L^2(T)} \\ && \qquad\qquad+\|\lambda-\lambda_h\|_{L^2(T)} \Big\}h_e^{-1}|\theta_1\theta_2|_{H^1(\mathcal{T}_e)},
\end{eqnarray*}
therein, using \eqref{edge:1} and \eqref{edge:2}, we have
\begin{eqnarray}
|\theta_1\theta_2|_{H^1(\mathcal{T}_e)}\leq |\theta_1|_{\infty,T_{\pm}}|\theta_2|_{1,T_{\pm}}+|\theta_1|_{1,T_{\pm}}|\theta_2|_{\infty,T_{\pm}}\lesssim \Big( h_e\beta^2\Big\|\jump{\frac{\partial^2 {y}_h^*}{\partial n^2} }\Big\|^2_{L^2(e)}\Big)^{\frac{1}{2}}.
\end{eqnarray}
Finally, the estimate \eqref{eq:eff3} can be realized by using  \eqref{eq:eff1}.\\

\par\noindent
(iv)(\textbf{Local bound for $\eta_4$)} Let $e\in \cE_h^i$ be an interior edge sharing the elements $T_+$ and $T_{-}$ and $\cT_e=T_{+}\cup T_{-}$. 
Define $\theta_3\in\mathbb{P}_0(\cT_e)$, by assigning $\theta_3=\beta\jump{\frac{\partial(\Delta  {y}_h^*)}{\partial n} }$ on $e$ and $\theta_3$ satisfies:
\begin{eqnarray}
\|\theta_3\|_{L^2(T_{\pm})}\lesssim h_e^{\frac{1}{2}} \Big\|\jump{\frac{\partial(\Delta {y}_h^*)}{\partial n} }\Big\|_{L^2(e)}.\label{es:4:1}
\end{eqnarray}
Further, let  $\theta_4\in \mathbb{P}_8(\cT_e)$ satisfies the following:

(a) $\theta_4$ is positive on the edge $e$ and takes unit value at the mid point of the edge.

(b) $\theta_4$ vanishes up to first order on $(\partial T_{+}\cup \partial T_{-})\setminus e$.

Then,  $\theta_4$ satisfy
\begin{eqnarray}
h_e^{-1}\|\theta_4\|_{L^2(T_{\pm})}+\|\theta_4\|_{L^{\infty}(T_{\pm})}\lesssim 1.\label{es:4:2}
\end{eqnarray}
Define $\hat{\phi}_e=\theta_3\theta_4$ on $\cT_e$ and let $\hat{\phi} \in W$ be the extension of $\hat{\phi}_e$ by zero outside $\cT_e$. 
From \eqref{eq:IntF}, (\ref{kkt}), (\ref{Dkkt}), discrete trace inequality and inverse inequality, it follows that
\begin{align}
  \beta^2\Big\|\jump{\frac{\partial(\Delta {y}^*_h)}{\partial n} }\Big\|_{L^2(e)}^2&=\beta^2\int_{e}\jump{\frac{\partial(\Delta {y}_h^*)}{\partial n} }^2\,ds\leq \int_{e}\beta\jump{\frac{\partial(\Delta {y}_h^*)}{\partial n} }\hat{\phi}\,ds\notag\\
 &=\beta\Big\{\int_{\cT_e}\Delta^2 {y}_h^*\hat{\phi}\,dx-\sum_{T\in\cT_e}\int_{T}D^2{y}_h^*:D^2\hat{\phi}\,dx-\int_{e}\jump{\frac{\partial^2{y}^*_h}{\partial n^2} }\frac{\partial \hat{\phi}}{\partial n}\,ds\Big\}\nonumber\\
 &=\sum_{T\in \cT_e}\Big\{\beta\int_{T}D^2({y}^*-{y}_h^*):D^2\hat{\phi}\,dx-\int_T(y_d+\mu_h-{y}_h^*)\,\hat{\phi}\,dx\nonumber\\
&~~~~+\int_{T}(\mu_h-\mu)\hat{\phi}\,dx+\int_{T}({y}^*-{y}_h^*)\hat{\phi}\,dx+\int_{T}(\lambda-\lambda_h)\Delta_h \hat{\phi}\,dx\Big\}\notag\\
&~~~~~ -\beta\int_{e}\jump{\frac{\partial^2{y}_h^*}{\partial n^2} }\frac{\partial \hat{\phi}}{\partial n}\,ds\nonumber\\
&\lesssim \big(\beta h^{-2}_e |{y}^*-{y}_h^*|_{H^2(\cT_e)}+\|{y}^*-{y}_h^*\|_{L^2(\cT_e)}+\|y_d+\mu_h-{y}_h^*\|_{L^2(\cT_e)}\notag\\&~~~~+\|\mu-\mu_h\|_{L^2(\cT_e)}+ h_e^{-2}\|\lambda-\lambda_h\|_{L^2(\cT_e)}\big)|\hat{\phi}|_{L^2(\cT_e)}+\beta\Big\|\jump{\frac{\partial^2{y}_h^*}{\partial n^2} }\Big\|_{L^2(e)}\Big\|\frac{\partial \hat{\phi}}{\partial n}\Big\|_{L^2(e)}\notag\\
&\lesssim \Bigg(h_e^{-2}\{\beta|{y}^*-{y}_h^*|_{H^2(\cT_e)}+h^2_e\|{y}^*-{y}_h^*\|_{L^2(\cT_e)}+h_e^2\|y_d+\mu_h-{y}_h^*\|_{L^2(\cT_e)}\notag\\&~~~~+h_e^2 \|\mu-\mu_h\|_{L^2(\cT_e)}+\|\lambda-\lambda_h\|_{L^2(\cT_e)}\}+\beta h_e^{\frac{-3}{2}}\Big\|\jump{\frac{\partial^2{y}_h^*}{\partial n^2} }\Big\|_{L^2(e)}\Bigg)\|\hat{\phi}\|_{L^2(\cT_e)}\notag\\
&\lesssim \Bigg(h_e^{-2}\{\beta|{y}^*-{y}_h^*|_{H^2(\cT_e)}+h^2_e\|{y}^*-{y}_h^*\|_{L^2(\cT_e)}+h_e^2\|y_d+\mu_h-\bar{y}_h\|_{L^2(\cT_e)}\notag\\&~~~~+h_e^2\|\mu-\mu_h\|_{L^2(\cT_e)}+\|\lambda-\lambda_h\|_{L^2(\cT_e)}\}+\beta h_e^{\frac{-3}{2}}\Big\|\jump{\frac{\partial^2{y}_h^*}{\partial n^2} }\Big\|_{L^2(e)}\Bigg)\beta h_e^{\frac{1}{2}}\Big\|\jump{\frac{\partial (\Delta {y}_h^*)}{\partial n}}\Big\|_{L^2(e)},
\end{align}
where in the last step we have used (\ref{es:4:1}) and (\ref{es:4:2}). Hence,
\begin{align}
\beta{h_e^{3/2}\Bigg\|\jump{\frac{\partial (\Delta {y}_h^*)}{\partial n} } \Bigg\|_{L^2(e)}}&\lesssim \sum_{T\in\cT_e}\Big(\|{y}^*-{y}_h^*\|_{2,T}+h_T^2\|y_d+\mu_h-{y}_h^*\|_{L^2(T)}+h_T^2 \|\mu-\mu_h\|_{L^2(T)}\notag\\&~~~~~+\|\lambda-\lambda_h\|_{L^2(T)}\Big)+h_e^{1/2}\beta\Big\|\jump{\frac{\partial^2{y}_h^*}{\partial n^2} }\Big\|_{L^2(e)}.
\end{align}
Finally, we obtain the bound \eqref{eq:eff4} by a use of \eqref{eq:eff1} and \eqref{eq:eff3}.\\

\par\noindent
(v)\textbf{(Local bound for $\eta_5$)} Let $\hat{b}_T$ be a polynomial bubble function vanishing up to the second order on $\partial T$. Let $\Delta\psi_T=\lambda_h\Delta \hat{b}_T$ and $\tilde{\psi} \in W$ be the extension of $\psi_T$ by zero to $\Omega$. 
\par \noindent
In view of \eqref{eq:IntF}, $\int_T D^2{y}^*_h:D^2{\psi}_T\,dx=0$. Therefore, using (\ref{kkt}) we have,



\begin{align*}
\|\lambda_h\|^2_{L^2(T)}&\lesssim \int_{T}\lambda_h\Delta\psi_T\,dx
=\int_{\Omega}(\lambda_h-\lambda)\Delta \tilde{\psi}\,dx+\int_{\Omega}\lambda\Delta \tilde{\psi}\,dx\\
&=\int_{\Omega}(\lambda_h-\lambda)\Delta \tilde{\psi}\,dx+\int_{\Omega}y_d\tilde{\psi}\,dx+\int_{\Omega}\mu\tilde{\psi}\,dx-\mathcal{A}({y}^*,\tilde{\psi})\\
&=\int_{T}(\lambda_h-\lambda)\Delta \tilde{\psi}\,dx+ \int_{T}(y_d+\mu_h-{y}^*_h)\tilde{\psi}\,dx+\int_{T}(\mu-\mu_h)\tilde{\psi}\,dx\\
&-\int_{T}({y}^*-{y}^*_h)\tilde{\psi}\,dx-\beta \int_{T}D^2({y}^*-{y}^*_h):D^2\tilde{\psi}\,dx\\
&\lesssim  \Big(\|\lambda-\lambda_h\|_{L^2(T)}+h_T^2\|\mu-\mu_h\|_{L^2(T)}+h^2_T\|y_d+\mu_h-{y}^*_h\|_{L^2(T)}  \\ &\quad+\|{y}^*-{y}_h^*\|_{2,T} \Big)\|\psi_T\|_{H^2(T)}\\
&\lesssim  \Big( \|\lambda-\lambda_h\|_{L^2(T)}+h_T^2\|\mu-\mu_h\|_{L^2(T)}+h^2_T\|y_d+\mu_h-{y}^*_h\|_{L^2(T)}  \\& \quad+\|{y}^*-{y}_h^*\|_{2,T} \Big)\|\lambda_h\|_{L^2(T)}
\end{align*} 

\par \noindent
where in obtaining the second last estimate, we have used Poincar{\'e} inequality with scaling arguments. Finally, we get the desired estimate by taking into account \eqref{eq:eff1}-\eqref{eq:eff4}.

\end{proof}

%

\section{ Adaptive FEM for OCPs with integral state constraint and poinwise control constraints} \label{sec:OCPs}
This section is devoted to the a posteriori error analysis of OCPs with integral state constraint and pointwise control constraints. 
We consider the following minimization problem: find $({y}^*,{u}^*)\in K$, such that
\begin{eqnarray}
 ({y}^*,{u}^*)= \displaystyle argmin_{(y,u)\in {K}} \Big(\frac{1}{2}\|y-y_d\|^2 +\frac{\beta}{2}\|u\|^2  \Big)\label{intro:functional1}
\end{eqnarray}
 subject to the constraints 
\begin{eqnarray}
\begin{cases}
\int_{\Omega}\nabla y \cdot \nabla w~dx = \int_{\Omega} u w~dx,\;\;\;\;\forall w \in H^1_0(\Omega)\\
\int_{\Omega}y\,dx \geq \delta_3,\;\;\;\;\; \\
u_a\leq u \leq u_b ~~\text{a.e. in}~~ \Omega,
\end{cases}\label{intro:state:cons1}
\end{eqnarray}
where ${ {K}} = H^1_0(\Omega)\times L^2(\Omega)$ and $\delta_3$ is a constant.
The functions $u_a$, $u_b$ are assumed to satisfy $(i)\,u_a,\;u_b\in W^{1,\infty}(\Omega)$, $(ii)\,u_a<u_b$ on $\bar{\Omega}$.

\noindent
As discussed in Section \ref{sec:intro}, we then rewrite this optimization problem into a reduced minimization problem involving only the state variable. Analogously to \eqref{eq:MP1}, the reduced optimal control problem is to find ${y}^* \in \tcK$ such that
 \begin{eqnarray}\label{eq:QP1}
  {y}^*= argmin_{y\in {{\tcK}}} \Big(\frac{1}{2}\mathcal{A}(y,y)-(y_d,y)\Big),
 \end{eqnarray}
where the bilinear form $\mathcal{A}(,\cdot,)$ is same as in \eqref{intro:state:modi}  and set $\tcK$ is defined as
\begin{eqnarray}
 \tcK=\{w\in W:\,\int_{\Omega}w\,dx\geq \delta_3 \;\;\text{and}\;\;u_a\leq -\Delta w\leq u_b\;\;\text{a.e. in}\;\Omega\}.\label{intro:Tk}
 \end{eqnarray}
 \noindent
 We assume the following Slater condition:
  there exists $ y\in W$ such that  $\int_{\Omega}y\;dx >\delta_3$ and $u_a \leq -\Delta y \leq u_b$.
\noindent
  Thus, the closed convex set $\tcK$ is nonempty. The minimizer of \eqref{eq:QP1} have the characterization in terms of the solution of the following fourth order variational inequality: find ${y}^* \in \tcK$ satisfying
 \begin{equation}
   \mathcal{A}({y}^*,w-{y}^*) \geq \int_{\Omega}y_d(w-{y}^*)\,dx\;\;\;\;\forall w\in {\tcK}.\label{QP2}
  \end{equation}
\noindent
The following (generalized) Karush-Kuhn-Tucker conditions hold (see \cite{ito,luenberger}): there exist $\lambda\in L^2(\Omega)$ and ${\mu}\in\mathbb{R}$  such that 
 \begin{eqnarray}
 \beta\int_{\Omega}(\Delta {y}^*)(\Delta w)\,dx+\int_{\Omega}{y}^* w\,dx= \int_{\Omega}y_d w\,dx -\int_{\Omega}\lambda (\Delta w)\,dx+\int_{\Omega} \mu w\,dx\label{kkt1}
 \end{eqnarray} 
 for all $w\in W$ together with the complementarity conditions
 \begin{eqnarray}
 \lambda \geq 0\;\;\;\;&&\;\;\text{if}\;\;-\Delta {y}^*=u_a,\label{lambda:1}\\
 \lambda \leq 0\;\;\;\;&&\;\;\text{if}\;\;-\Delta {y}^*=u_b,\label{lambda:2}\\
 \lambda =0\;\;\;\;&&\;\;\text{otherwise},\label{lambda:3}\\
 \mu\geq 0\;\;\;\;\;&&\;\;\text{if}\;\;\int_{\Omega}  {y}^*\,dx=\delta_3,\\
 \mu=0\;\;\;\;&&\;\;\text{if}\;\;\int_{\Omega} {y}^*\,dx>\delta_3.
 \end{eqnarray}
 \noindent
 The adjoint state $p\in H^1_0(\Omega)$ associated to the problem \eqref{intro:functional1} -\eqref{intro:state:cons1} is given by
 \begin{eqnarray}
 \int_{\Omega} \nabla p \cdot \nabla w\,dx=\int_{\Omega}(y^{*}-y_d)w\,dx-\int_{\Omega} \mu w\,dx \quad  \forall~ w\in  H^{1}_0(\Omega).\label{adjointN}
 \end{eqnarray}
\par
\noindent 
\subsection{Discrete Problem}
The discretization of \eqref{eq:QP1} is to find ${y}^*_h \in \tcK_h$ such that
\begin{eqnarray}
{y}^*_h=argmin_{y_h\in \tcK_h}\Big[\frac{1}{2}\mathcal{A}_h(y_h,y_h)-(y_d,y_h)\Big],\label{eq:DP1}
\end{eqnarray}
where 
\begin{equation}
\tcK_h = \{w_h\in W_h: \int_{\Omega} w_h\,dx\geq \delta_3\;\;\text{and}\;\;Q_h  u_a\leq Q_h(-\Delta_h w_h)\leq Q_h u_b\}\label{KhN}.
\end{equation}
\begin{remark} Owing to the property \eqref{eq:IntC1} and \eqref{interpo:2} of $I_h$, we have $I_h \tcK \subset \tcK_h$.
\end{remark}
\noindent
As in the continuous case, the minimizer of \eqref{eq:DP1} can be characterized by the solution of the following variational inequality: find ${y}^*_h \in \tcK_h$ such that
\begin{eqnarray}
\mathcal{A}_h({y}^*_h,w_h-{y}^*_h)\geq (y_d,w_h-{y}^*_h)\;\;\;\;\;\forall w_h\in \tcK_h.\label{DP2}
\end{eqnarray} 
where the bilinear form $\mathcal{A}_h(,\cdot,)$ is defined in \eqref{discrete:variational:inequality}. It can be easily checked that the discrete problem \eqref{DP2} is well-posed.
\par
\noindent
The Karush-Kuhn-Tucker conditions for the discrete problem  \cite{ito,luenberger} is given as follows: there exist $\lambda_h\in \mathbb{P}_0(\cT_h)$ and ${\mu}_h\in\mathbb{R}$  such that 
 \begin{eqnarray}
 \mathcal{A}_h({y}^*_h, w_h)= \int_{\Omega}y_d w_h\,dx-\int_{\Omega}\lambda_h (\Delta_h w_h)\,dx +\int_{\Omega} \mu_h w_h\,dx,\quad \forall w_h \in W_h  \label{Dkkt:new}
 \end{eqnarray} 
 together with the complementary conditions
 \begin{eqnarray}
 \lambda_h \geq 0\;\;\;\;&&\;\;\text{on} ~T \in \cT_h~ \text{such that}\;\;Q_T(-\Delta {y}^*_h)=Q_T(u_a),\label{Dlambda:1}\\
 \lambda_h \leq 0\;\;\;\;&&\;\;\text{on} ~T \in \cT_h~ \text{such that}\;\;Q_T(-\Delta {y}^*_h)=Q_T(u_b),\label{Dlambda:2}\\
 \lambda_h =0\;\;\;\;&&\;\;\text{otherwise},\label{Dlambda:3}\\
 \mu_h\geq 0\;\;\;\;&&\;\;\text{if}\;\;\int_{\Omega}{y}^*_h\,dx=\delta_3,\\
 \mu_h=0\;\;\;\;&&\;\;\text{if}\;\;\int_{\Omega}{y}^*_h\,dx>\delta_3.
 \end{eqnarray}
In the following, we derive the reliability estimates of the estimator $\eta_h$ for the error $\|{y}^*-{y}_h^*\|_h$. For this we introduce the following auxiliary problem: let $\tilde{z}_h\in W$  be the solution of 
\begin{eqnarray}
\mathcal{A}(\tilde{z}_h,w)=\int_{\Omega}y_d w\,dx+\int_{\Omega}\lambda_h(-\Delta_h w)\,dx+\int_{\Omega}\mu_h w\,dx,\;\;\;\;\forall w\in W.\label{aux:prob:2}
\end{eqnarray}
The well-posedness of \eqref{aux:prob:2} is ensured by Lax-Milgram lemma \cite{Ciarlet:1978:FEM}.
Now, we proceed to establish the reliability of the error estimator for the error in solution ${y}^*$. 
\begin{theorem}\label{thm:2}
Let ${y}^*_h$ and $\tilde{z}_h$ be the solutions of (\ref{Dkkt:new}) and (\ref{aux:prob:2}), respectively. Then, 
\begin{eqnarray}
\|{y}^*_h-\tilde{z}_h\|_h\lesssim {\eta}_h,
\end{eqnarray}
where $\eta_h$ is defined in (\ref{etah}).
\end{theorem}

\begin{proof}
Let $\phi=\tilde{z}_h-E_h{y}^*_h \in W$ and $\phi_h=I_h \phi \in W_h$, a use of coercive property of the bilinear form $\mathcal{A}(\cdot,\cdot)$, (\ref{aux:prob:2}) and (\ref{Dkkt:new}) leads to
\begin{eqnarray}
\|\phi\|_h^2&\lesssim & \mathcal{A}(\tilde{z}_h,\phi)-\mathcal{A}_h(E_h{y}_h^*,\phi)\nonumber\\
             &\lesssim & \int_{\Omega}y_d\phi\,dx+\int_{\Omega}\mu_h\phi\,dx-\mathcal{A}_h({y}^*_h,\phi)+\int_{\Omega}\lambda_h(-\Delta_h\phi)\,dx\nonumber\\&&+\mathcal{A}_h({y}^*_h-E_h{y}^*_h,\phi) \nonumber\\
             &\lesssim &\int_{\Omega}y_d(\phi-\phi_h)\,dx-\mathcal{A}_h({y}^*_h,\phi-\phi_h)\,dx+\int_{\Omega}\mu_h(\phi-\phi_h)\,dx\nonumber\\&&-\int_{\Omega}\lambda_h( \Delta_h(\phi-\phi_h))+\mathcal{A}_h({y}^*_h-E_h{y}^*_h,\phi)\nonumber\\
&=:&I_1+I_2+I_3+I_4+I_5.\label{apos:pro:1}
\end{eqnarray}
\par \noindent
Note that, $I_1+I_2+I_3$ and $I_5$ can be estimated following same arguments as in the Theorem \ref{thm:rel}. 
\noindent
Now, it remains to estimate $I_4$. 
An application of the Cauchy Schwarz inequality and Lemma \ref{lem:approx_proj}, yields
 \begin{eqnarray*}
|I_4| & \lesssim & \sum_{T \in \cT_h}  |\int_T \lambda_h\Delta(\phi_h- \phi)~dx|
 \lesssim  \sum_{T \in \cT_h} h_T|\lambda_h\| \| \Delta(\phi_h- \phi) \|_{L^2(T)}\\
& \lesssim &\sum_{T \in \cT_h} h_T|\lambda_h\| |\phi|_{H^2(T)} 
\lesssim {\eta}_5 \|\phi\|_{h}.
 \end{eqnarray*}

\par \noindent
Combining all the estimates together with (\ref{apos:pro:1}),   we get 
\begin{eqnarray*}
\|\phi\|_h = \|\tilde{z}_h-E_h{y}^*_h\|_h\lesssim {\eta}_h.
\end{eqnarray*}
An use of the triangle inequality $\|\tilde{z}_h-{y}^*_h\|_h\leq \|\tilde{z}_h-E_h{y}^*_h\|_h+\|E_h{y}^*_h-{y}^*_h\|_h$, in view of Lemma \ref{Eh:lemma} leads to the desired estimate. 
\end{proof}
\begin{theorem} \label{main:thm:2}
Let ${y}^*$ and ${y}^*_h$ be solutions of variational inequalities \eqref{QP2} and \eqref{DP2}, respectively. Then, it holds that

\begin{eqnarray*}
\|{y}^*-{y}_h^*\|_h&\lesssim& \Big({\eta}_h
+\displaystyle\sum_{T\in\Omega_1 \cup \Omega_2} \|\lambda\|_{L^2(T)}^{\frac{1}{2}} \Big( \sum_{e \in \mathcal{E}_T}\frac{1}{h_e}\Big\|\jump{\frac{\partial y_h^*}{\partial n}}\Big\|_{L^2(e)}^2 \Big)^{\frac{1}{2}} \nonumber\\ &&\qquad+\Big(\int_{\Omega_1}\lambda(\Delta {y}^*_h+u_a)\,dx \Big)^{\frac{1}{2}} +\Big(\int_{\Omega_2}\lambda(\Delta_h {y}^*_h+u_b)\,dx \Big)^{\frac{1}{2}}\Big).
\end{eqnarray*}


\end{theorem}

\begin{proof}
Set $\phi={y}^*-E_h{y}^*_h\in W$ and let $\phi_h\in W_h$. As in Theorem  \ref{thm:rel}, using coercivity of the bilinear form $\mathcal{A}(\cdot,\cdot)$, we get
\begin{eqnarray}
\|{y}^*-E_h{y}^*_h\|^2_{h}&\lesssim &\mathcal{A}({y}^*-E_h{y}^*_h,\phi)\nonumber\\
&=&\mathcal{A}(\tilde{z}_h-E_h{y}^*_h,\phi)+\mathcal{A}({y}^*,\phi)-\mathcal{A}(\tilde{z}_h,\phi).\label{eq:main:1:1}
\end{eqnarray}
In view of (\ref{kkt1}) and (\ref{Dkkt:new}), the term $\mathcal{A}({y}^*,\phi)-\mathcal{A}(\tilde{z}_h,\phi)$ satisfies  
\begin{eqnarray}
\mathcal{A}({y}^*,\phi)-\mathcal{A}(\tilde{z}_h,\phi)=\int_{\Omega}(\lambda-\lambda_h)(-\Delta \phi)\,dx+\int_{\Omega}(\mu-\mu_h)\phi\,dx.\label{main:eq:1:1}
\end{eqnarray}
Following the same arguments as in Theorem \ref{thm:rel}, we obtain
\begin{align}
\int_{\Omega} (\mu-\mu_h) \phi\,dx\leq 0.\label{main:eq:1:2}
\end{align}

\par \noindent
Now, to estimate the term $\int_{\Omega}(\lambda-\lambda_h)(-\Delta\phi)\,dx$,
we split it as
\begin{eqnarray}
\int_{\Omega}(\lambda-\lambda_h)(-\Delta\phi)\,dx=\int_{\Omega_1}(\lambda-\lambda_h)(-\Delta\phi)\,dx+\int_{\Omega_2}(\lambda-\lambda_h)(-\Delta\phi)\,dx,\label{:q:1}
\end{eqnarray}
where  $\Omega_1$ and $\Omega_2$ are discrete control contact sets defined by
 \begin{align*}
 \Omega_1 &= \{T \in \mathcal{T}_h : Q_T(-\Delta y^*_h)=Q_T(u_a)\}, \\
 \Omega_2  &= \{T \in \mathcal{T}_h : Q_T(-\Delta y^*_h)=Q_T(u_b)\}.
 \end{align*}
 The first term of the right hand side of equation (\ref{:q:1}) can be estimates as follows.
\begin{eqnarray}
\int_{\Omega_1}(\lambda-\lambda_h)(-\Delta\phi)\,dx=\int_{\Omega_1}\lambda_h(\Delta\phi)\,dx-\int_{\Omega_1}\lambda(\Delta\phi)\,dx\nonumber\\
=\int_{\Omega_1}\lambda_h(\Delta {y}^*-\Delta E_h{y}^*_h)\,dx-\int_{\Omega_1}\lambda(\Delta {y}^*-\Delta E_h{y}^*_h)\,dx.\label{es:1}
\end{eqnarray}
Using $\lambda_h \geq 0$, $u_a\leq -\Delta {y}^*$, relation (\ref{eq:EnrichPP})  and $\displaystyle\sum_{T\in\Omega_1}\lambda_h\int_{T}(u_a+\Delta_h{y}^*_h)\,dx=0$, we  find
\begin{align}
\int_{\Omega_1}\lambda_h(\Delta {y}^*-\Delta E_h{y}^*_h)\,dx&=\int_{\Omega_1}\lambda_h(\Delta{y}^*+u_a)\,dx-\int_{\Omega_1}\lambda_h(u_a+\Delta E_h{y}^*_h)\,dx\nonumber\\&=\int_{\Omega_1}\lambda_h(\Delta{y}^*+u_a)\,dx-\int_{\Omega_1}\lambda_h(u_a+\Delta_h{y}^*_h)\,dx\nonumber\\&\hspace{0.5cm}-\int_{\Omega_1}\lambda_h(\Delta_h{y}^*_h-\Delta E_h{y}^*_h)\,dx\nonumber\\
&\leq 0.\label{es:2}
\end{align}
In view of \eqref{lambda:1} and Lemma \ref{Eh:lemma}, we have
\begin{align}
-\int_{\Omega_1}\lambda(\Delta {y}^*-\Delta E_h{y}^*_h)\,dx=&\int_{\Omega_1}\lambda(\Delta E_h{y}^*_h-\Delta{y}^*)\,dx\nonumber\\
&=\int_{\Omega_1}\lambda(\Delta E_h{y}^*_h-\Delta_h{y}^*_h)\,dx+\int_{\Omega_1}\lambda (\Delta_h{y}^*_h+u_a)\,dx\nonumber\\& \quad -\int_{\Omega_1}\lambda (\Delta_h{y}^*+u_a)\,dx \notag\\
& \lesssim \displaystyle\sum_{T\in\Omega_1} \|\lambda\|_{L^2(T)} \sum_{e \in \mathcal{E}_T}\frac{1}{h_e}\Big\|\jump{\frac{\partial y_h^*}{\partial n}}\Big\|_{L^2(e)}^2+\int_{\Omega_1}\lambda(\Delta {y}^*_h+u_a)\,dx.\label{es:31}
\end{align}
Combining (\ref{es:1}),(\ref{es:2}) and (\ref{es:31}), we get
\begin{align}
\int_{\Omega_1}(\lambda-\lambda_h)(-\Delta\phi)\,dx& \lesssim \displaystyle\sum_{T\in\Omega_1} \|\lambda\|_{L^2(T)} \sum_{e \in \mathcal{E}_T}\frac{1}{h_e}\Big\|\jump{\frac{\partial y_h^*}{\partial n}}\Big\|_{L^2(e)}^2+\int_{\Omega_1}\lambda(\Delta {y}^*_h+u_a)\,dx.\label{es:4}
\end{align}
Repeating the similar arguments, we estimate the second term of the right hand side of equation (\ref{:q:1})
\begin{align}
\int_{\Omega_2}(\lambda-\lambda_h)(-\Delta\phi)\,dx& \lesssim \displaystyle\sum_{T\in\Omega_2} \|\lambda\|_{L^2(T)} \sum_{e \in \mathcal{E}_T}\frac{1}{h_e}\Big\|\jump{\frac{\partial y_h^*}{\partial n}}\Big\|_{L^2(e)}^2+\int_{\Omega_2}\lambda(\Delta_h {y}^*_h+u_b)\,dx.\label{es:5}
\end{align}
A use of  (\ref{es:4}) and (\ref{es:5}) in \eqref{:q:1} yields
\begin{align}
\int_{\Omega}(\lambda-\lambda_h)(-\Delta\phi)\,dx& \lesssim \displaystyle\sum_{T\in\Omega_1 \cup \Omega_2} \|\lambda\|_{L^2(T)} \sum_{e \in \mathcal{E}_T}\frac{1}{h_e}\Big\|\jump{\frac{\partial y_h^*}{\partial n}}\Big\|_{L^2(e)}^2+\int_{\Omega_1}\lambda(\Delta {y}^*_h+u_a)\,dx \nonumber\\ &\qquad+\int_{\Omega_2}\lambda(\Delta_h {y}^*_h+u_b)\,dx.\label{es:6}
\end{align}
Combining (\ref{main:eq:1:1}), (\ref{main:eq:1:2}) and (\ref{es:6}), we have
\begin{eqnarray*}
\mathcal{A}({y}^*,\phi)-\mathcal{A}(\tilde{z}_h,\phi)& \lesssim& \displaystyle\sum_{T\in\Omega_1 \cup \Omega_2} \|\lambda\|_{L^2(T)} \sum_{e \in \mathcal{E}_T}\frac{1}{h_e}\Big\|\jump{\frac{\partial y_h^*}{\partial n}}\Big\|_{L^2(e)}^2+\int_{\Omega_1}\lambda(\Delta {y}^*_h+u_a)\,dx \nonumber\\ &&\qquad+\int_{\Omega_2}\lambda(\Delta_h {y}^*_h+u_b)\,dx.
\end{eqnarray*}
Using   the continuity of  bilinear  form  and the Young's inequality in (\ref{eq:main:1:1}) leads to
\begin{eqnarray}
\|{y}^*-E_h{y}^*_h\|_{h}^2 &\lesssim& \Big(\|\tilde{z}_h-E_h{y}^*_h\|_h^2+\displaystyle\sum_{T\in\Omega_1 \cup \Omega_2} \|\lambda\|_{L^2(T)} \sum_{e \in \mathcal{E}_T}\frac{1}{h_e}\Big\|\jump{\frac{\partial y_h^*}{\partial n}}\Big\|_{L^2(e)}^2 \nonumber\\ &&\qquad+\int_{\Omega_1}\lambda(\Delta {y}^*_h+u_a)\,dx +\int_{\Omega_2}\lambda(\Delta_h {y}^*_h+u_b)\,dx\Big).\label{es:7}
\end{eqnarray}
Finally, a use of triangle inequality, (\ref{es:7}), Theorem \ref{thm:2} and Lemma \ref{Eh:lemma} gives
\begin{eqnarray*}
\|{y}^*-{y}^*_h\|_h^2&\lesssim& \Big({\eta}_h^2+\displaystyle\sum_{T\in\Omega_1 \cup \Omega_2} \|\lambda\|_{L^2(T)} \sum_{e \in \mathcal{E}_T}\frac{1}{h_e}\Big\|\jump{\frac{\partial y_h^*}{\partial n}}\Big\|_{L^2(e)}^2 \nonumber\\ &&\qquad+\int_{\Omega_1}\lambda(\Delta {y}^*_h+u_a)\,dx +\int_{\Omega_2}\lambda(\Delta_h {y}^*_h+u_b)\,dx\Big).
\end{eqnarray*}
This completes the proof.

\end{proof}
\par\noindent
We would like to remark here that, in Theorem \ref{main:thm:2} the estimate is not a genuine a posteriori error estimate because of the presence of $\lambda$ in the right hand side, but it is useful in realizing the asymptotic convergence  of the adaptive algorithm. 
\noindent
Now, following the idea of Lemma \ref{lemma:adjoint} and Theorem \ref{thm:3}, we can estimate the error in Lagrange multipliers, hence we state the result omitting details of the proof.
\begin{lemma} It holds that,
\begin{eqnarray*}
|\mu-\mu_h|+\|\lambda-\lambda_h\|_{L^2(\Omega)}&\lesssim& \eta_h.\label{est:3}
\end{eqnarray*}
\end{lemma}
\par \noindent
The following local efficiency estimates can be proved using bubble function techniques as in Theorem \ref{thm:EffEstimates}.

\begin{theorem}
There exists a positive constant $C>0$ depending on the shape regularity of $\mathcal{T}_h$ such that
\begin{eqnarray*}
h_T^2  \|y_d+\mu_h-{y}_h^*\|_{L^2(T)} &\lesssim& \|{y}^*-{y}^*_h\|_{2,T}+h_T^2\|\mu-\mu_h\|_{L^2(T)}+\|\lambda-\lambda_h\|_{L^2(T)}+Osc(y_d;T)\;\;\forall T\in\mathcal{T}_h, \label{eq:eff11}\\
{\beta}{h_e^{-1/2}}\Big\|\jump{\frac{\partial {y}^*_h}{\partial n} } \Big\|_{L^2(e)}&\lesssim& \sum_{T\in\mathcal{T}_e} \Big(\|{y}^*-{y}^*_h\|_{2,T}+h_T^2\|\mu-\mu_h\|_{L^2(T)}+\|\lambda-\lambda_h\|_{L^2(T)} \Big)\nonumber\\&&\qquad+Osc(y_d;\mathcal{T}_e)\;\;\;\;\forall e\in \mathcal{E}_h^i, \label{eq:eff21}\\
\beta h_e^{1/2}\Bigg\|\jump{\frac{\partial^2 {y}^*_h}{\partial n^2} } \Bigg\|_{L^2(e)}&\lesssim& \sum_{T\in\mathcal{T}_e} \Big(\|{y}^*-{y}^*_h\|_{2,T}+h_T^2\|\mu-\mu_h\|_{L^2(T)}+\|\lambda-\lambda_h\|_{L^2(T)} \Big)\nonumber\\&&\qquad+Osc(y_d;\mathcal{T}_e)\;\;\;\forall e\in \mathcal{E}_h^i,  \label{eq:eff31}\\
\beta{h_e^{3/2}\Bigg\|\jump{\frac{\partial (\Delta {y}_h^*)}{\partial n} } \Bigg\|_{L^2(e)}}&\lesssim& \sum_{T\in\mathcal{T}_e} \Big(\|{y}^*-{y}^*_h\|_{2,T}+h_T^2\|\mu-\mu_h\|_{L^2(T)}+\|\lambda-\lambda_h\|_{L^2(T)} \Big)\nonumber\\&&\qquad+Osc(y_d;\mathcal{T}_e);\;\;\;\forall e\in \mathcal{E}_h^i,  \label{eq:eff41}\\
{h_T |\lambda_h|}&\lesssim &\|{y}^*-{y}^*_h\|_{2,T}+h_T^2\|\mu-\mu_h\|_{L^2(T)}+\|\lambda-\lambda_h\|_{L^2(T)}+Osc(y_d;T)\;\;\forall T\in\mathcal{T}_h,\label{eq:eff51}
\end{eqnarray*}
where $\|w\|_{2,T}:=\beta|w|_{H^2(T)}+h_T^2\|w\|_{L^2(T)}$  for any $w \in H^2(\Omega, \mathcal{T}_h)$,
 $Osc(y_d;T):=h_T^2\|y_d-\bar{y}_d\|_{L^2(T)}^2$ with  $\bar{y}_d:=\frac{1}{|T|}\int_T y_d\,dx$ and $\mathcal{T}_e$ denotes the union of elements sharing the edge $e$.
\end{theorem}


\section{Numerical Assessments}\label{sec:NumTests}
In this section, we perform numerical experiments to illustrate the performance of the error estimators derived in Section \ref{sec:Apos} and Section \ref{sec:OCPs}. For this, we have considered four examples. The data of first example is for the purely integral state constraints, the second one is based on the purely integral control constraint, the third example consists of the integral state and integral control constraints and the last example concerns the integral state and pointwise control constraints. The discrete problem is solved using the primal-dual active set method \cite{BIK:1999:PrimalDual,kunisch2002,kunisch2003,ito}.
For the adaptive refinement, we use the following paradigm
\begin{equation*}
{\bf SOLVE}\longrightarrow  {\bf ESTIMATE} \longrightarrow {\bf
MARK}\longrightarrow {\bf REFINE}
\end{equation*}
 We compute the discrete state using the primal-dual active set algorithm in step 'SOLVE'. Thereafter in step  'ESTIMATE', we compute the error estimator on each element $T \in \mathcal{T}_h$ and use D\"{o}rfler  marking strategy with parameter $\theta=0.3$ to mark the elements for refinement. Finally, a new adaptive mesh is obtained by performing refinement using the newest vertex bisection algorithm.
Below, we consider various test examples.
\par \noindent
\begin{exam}\label{exam:1}
This example consists of the integral state constraints as active constraints \cite{yuan2009}. Here, we solve the following problem on $\Omega=(0,1)^2$ with $\beta=1$.
\begin{eqnarray}
\begin{cases}
\displaystyle\min_{y\in \mathcal{K}}\Big\{\frac{1}{2}\|y-y_d\|^2_{L_2(\Omega)}+\frac{\beta}{2}\|u\|^2_{L_2(\Omega)},\\
\text{s.t.}\\
 -\Delta y=f+u\;\;\;\;\text{in}\;\Omega,\\
y=0\;\;\text{on}\;\partial\Omega,\\
\int_{\Omega}y\,dx\geq \delta_2~~ \text{and}~~\int_{\Omega}u\,dx\geq \delta_1,
\end{cases}\label{num:f}
\end{eqnarray}
with the exact solution and the data as
\begin{align*}
p&=sin(2\pi x_1)sin(2\pi x_2)+\frac{3}{8}sin(2\pi x_1)sin(4\pi x_2),\\
y&=p,\\
y_d&=y+\Delta p-0.4,
\\f&=-\Delta y-u,
\\
\delta_2&=-0.4,\\
\delta_1&=0,\\
u&=\max\{\tilde{p}+ \beta \tilde{\delta_1}, 0 \}-p,
 \end{align*}
where $\tilde{p}=\frac{\int_{\Omega}p\,dx}{\int_{\Omega}1\,dx}$ and $\tilde{\delta_1}= \frac{\delta_1}{ \int_{\Omega} 1 dx}$.
\end{exam}
\noindent
 Figure  \ref{fig:Ex1_error} depicts convergence behavior of the error and the estimator with respect to the increasing number of degrees of freedom (DoFs).
 From this figure, it is evident that  both error and error estimator converge with optimal rate ($1/\sqrt{DoFs}$).  Figure  \ref{fig:Ex1_error} also ensures the reliability of the error estimator.  Figure \ref{fig:Ex1_effi} shows the efficiency indices, ensuring that the error estimator is efficient. 

\begin{figure}[h!]
\centering
\begin{subfigure}{.49\textwidth}
  \centering
  \includegraphics[width=8.5cm, height=7cm]{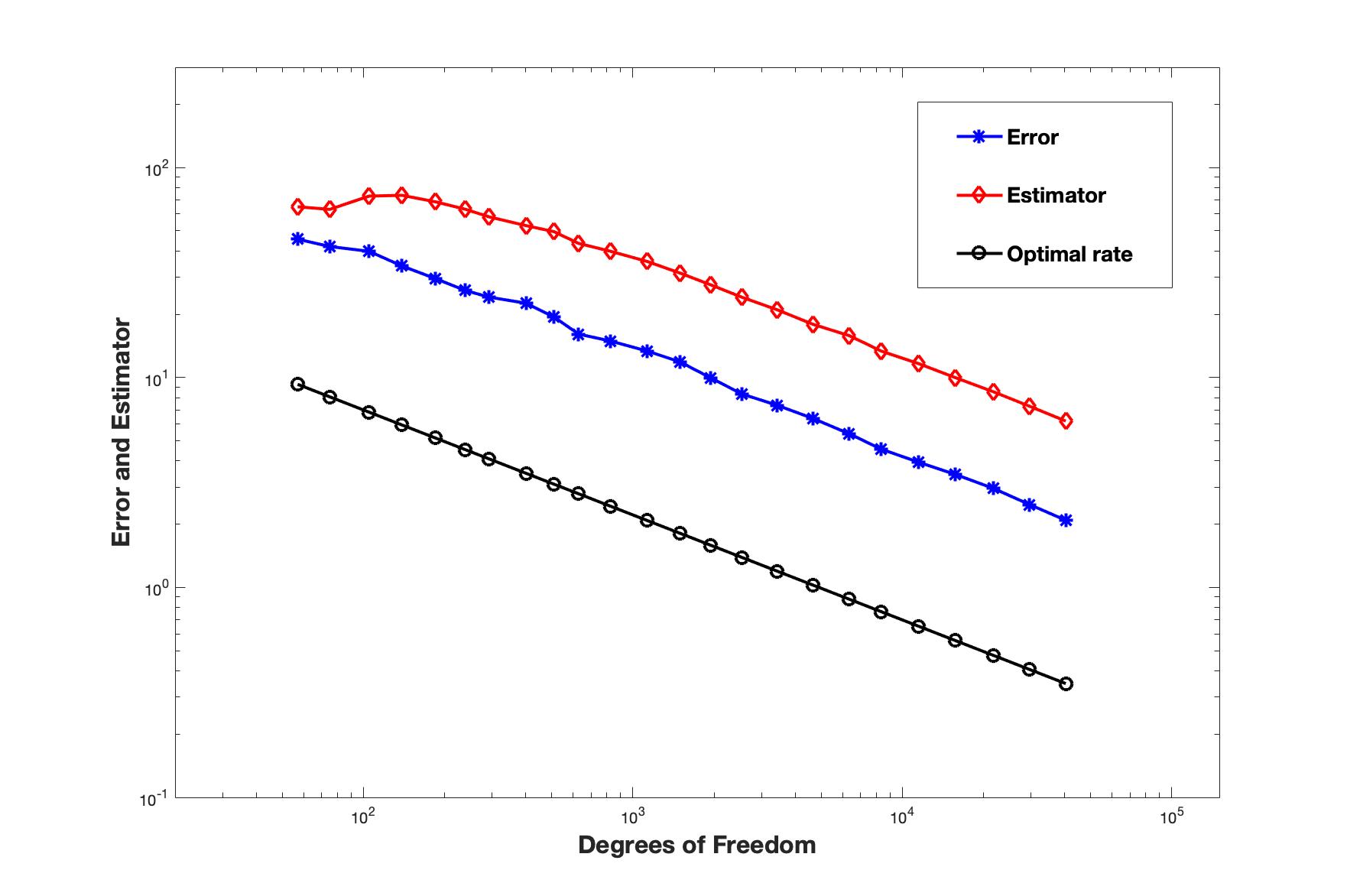} 
 \caption{Error  and Estimator}\label{fig:Ex1_error}
\end{subfigure}
\begin{subfigure}{.49\textwidth}
 \centering
 \includegraphics[width=8.5cm, height=7cm]{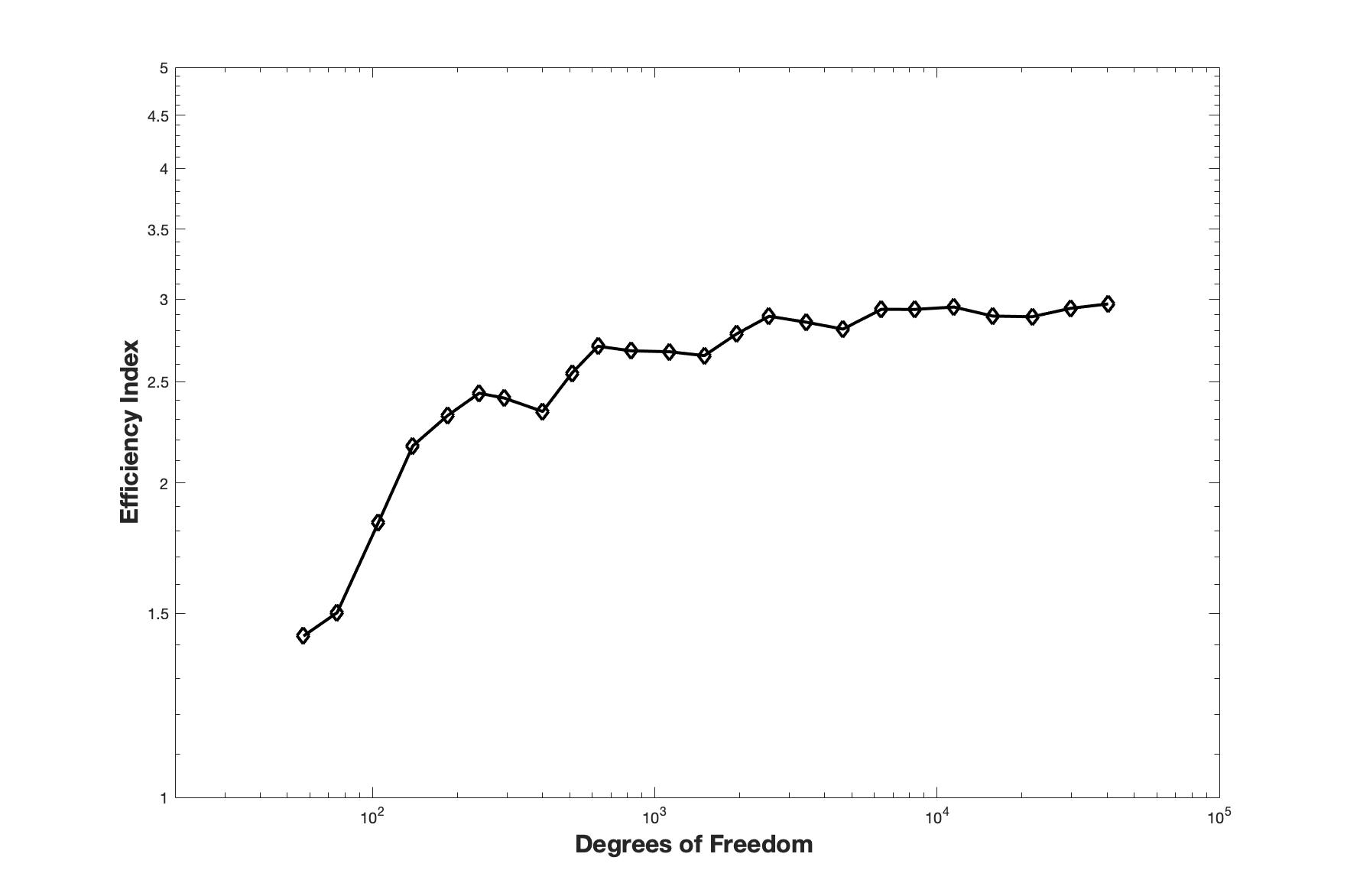}
   \caption{Efficiency Index}\label{fig:Ex1_effi}
\end{subfigure}
\caption{Error, estimator and efficiency index for Example \ref{exam:1}}
\end{figure}



%


%
%

%

\begin{exam}\label{Exam:2}
In this example \cite{ge2009}, we consider the optimal control problem (\ref{num:f}) with  purely integral  control constraints given on the domain $\Omega=(0,1)\times (0,1)$ with $\beta=1$ as follows
\begin{align*}
p&=sin(\pi x_1)sin(\pi x_2),\\
y&=2 \pi^2 p+y_d\\
y_d&=0,\\
f&=4 \pi^4 p+p-\frac{4}{\pi^2},\\
\delta_2&=100,\\
\delta_1&=0,\\
u&=\max\{\tilde{p}+ \beta \tilde{\delta_1}, 0 \}-p,
 \end{align*}
where $\tilde{p}=\frac{\int_{\Omega}p\,dx}{\int_{\Omega}1\,dx}$ and $\tilde{\delta_1}= \frac{\delta_1}{ \int_{\Omega} 1 dx}$.
\end{exam}
\par \noindent
For this example, the convergence behavior of the error and estimator is shown in Figure \ref{Ex_2:err}, which confirms that both error and estimator converges optimally and also that the estimator is reliable. The efficiency of the estimator is ensured by efficiency index depicted in Figure \ref{Ex_2:effi}.
%
%



\begin{figure}[h!]
\centering
\begin{subfigure}{.49\textwidth}
  \centering
  \includegraphics[width=8.5cm, height=7cm]{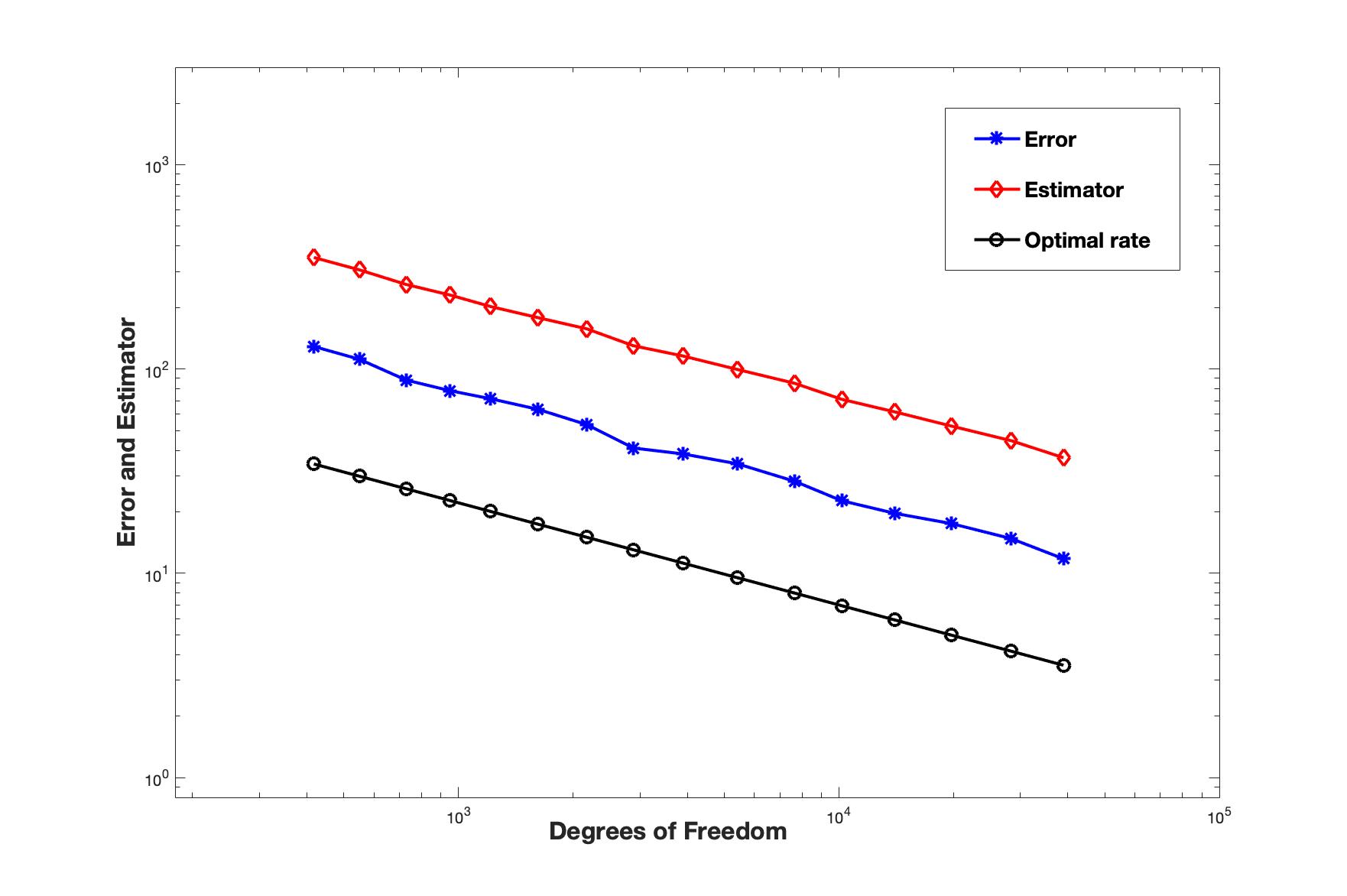}
  \caption{ Error and Estimator}\label{Ex_2:err}
\end{subfigure}%
\begin{subfigure}{.49\textwidth}
 \centering
 \includegraphics[width=8.5cm, height=7cm]{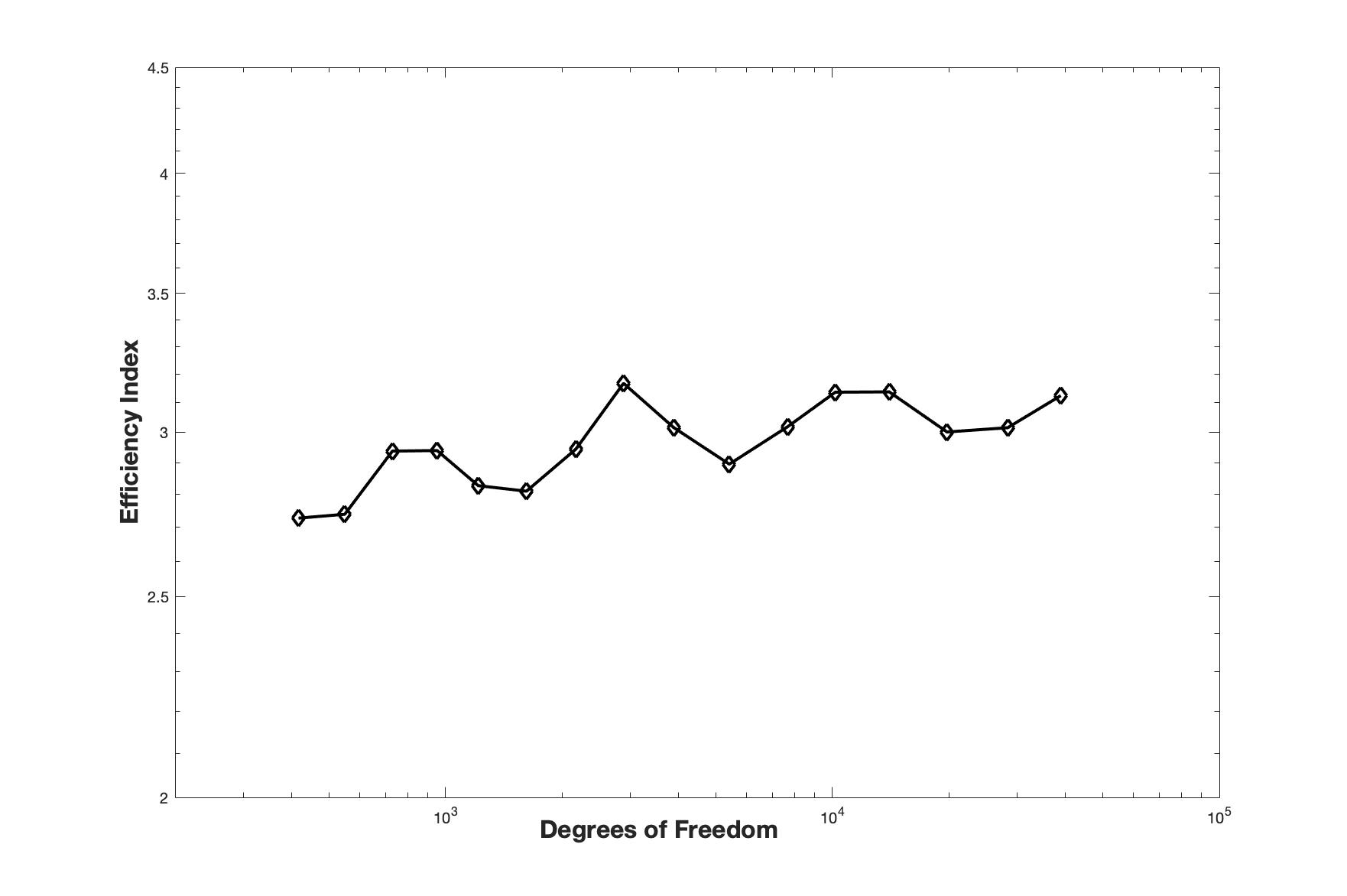}
   \caption{Efficiency Index}\label{Ex_2:effi}
\end{subfigure}
\caption{Error, estimator and efficiency index for Example \ref{Exam:2}}
\end{figure}



\begin{exam} \label{Exam:3}
In this example, we consider the OCP (\ref{num:f}) with integral state and integral control constraints  on the domain $\Omega=(-1,1)\times (-1,1)$ with  the following data \cite{lixin17}:
\begin{align*}
y&=\frac{-1}{2\pi^2}sin(\pi x_1)sin(\pi x_2),\\
p&=sin(\pi x_1)sin(\pi x_2),\\  
y_d&=-(2\pi^2+\frac{1}{2\pi^2})sin(\pi x_1)sin(\pi x_2)-0.6,\\
f&=0,\\
\delta_2&=0,\\ \delta_1 &=0,\\
\beta&=1,\\
u &=-p+max\{\tilde{p}+\beta \tilde{\delta}_1,0\},\\
\end{align*}
where $\tilde{p}=\int_{\Omega}p\,dx/\int_{\Omega}1\,dx$ and $\tilde{\delta}_1= \delta_1/ \int_{\Omega} 1\, dx$.
\end{exam}
\par \noindent
We plot the convergences histories for the error and the error estimator in Figure \ref{Ex_3:err} and the efficiency index in Figure
 \ref{Ex_3:effi}. These figures validates the reliability and efficiency of the error estimator together with the optimal convergence.
%
%



\begin{figure}[h!]
\centering
\begin{subfigure}{.5\textwidth}
  \centering
  \includegraphics[width=8.5cm, height=7cm]{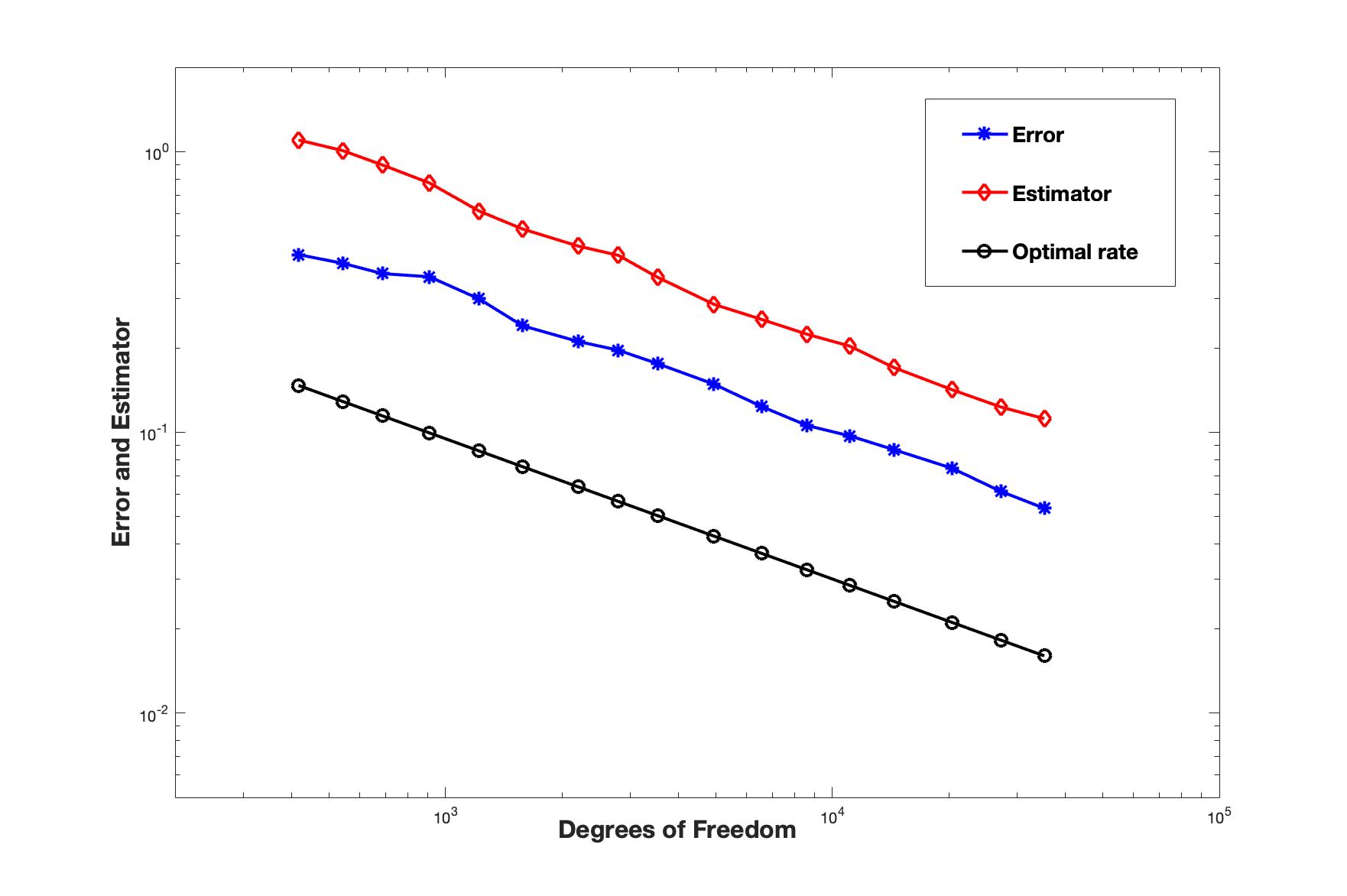}
  \caption{Error and Estimator}\label{Ex_3:err}
\end{subfigure}%
\begin{subfigure}{.5\textwidth}
 \centering
 \includegraphics[width=8.5cm, height=7cm]{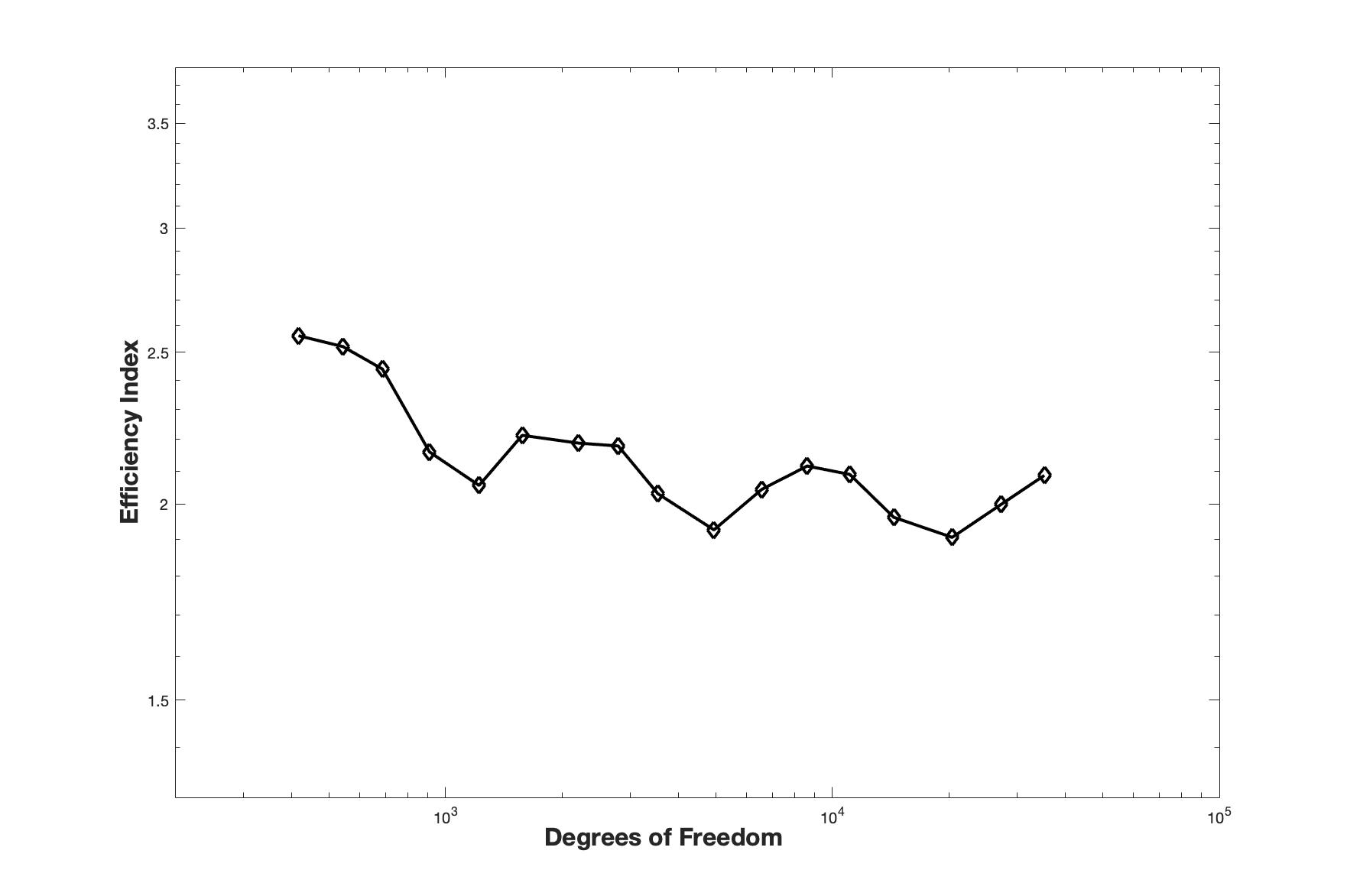}
   \caption{Efficiency Index}\label{Ex_3:effi}
\end{subfigure}
\caption{Error, estimator and efficiency index for Example \ref{Exam:3}}
\end{figure}


\begin{exam}\label{Exam:4}
In this example, we consider the problem (\ref{intro:functional1})-(\ref{intro:state:cons1}) with integral state and pointwise control constraints. The idea of this example is taken from \cite{lapin17}. Therein, the domain $\Omega=(0,1)\times (0,1)$ and the exact solution is not known.
\begin{align*}
y_d&=10(sin(\pi x_1)+sin(\pi x_2)),\\
\beta&=0.01,\\
\delta_3&=0,\\
u_a&=0\;\;\text{and}\;\; u_b=30.\\
\end{align*}
\end{exam}
\par \noindent
The behavior of error estimator is illustrated in Figure \ref{Ex_4:err} confirming the optimal convergence and realibility of the error estimator. The adaptive mesh at a certain refinement level is depicted in Figure \ref{Ex_4:ada}.\\
\begin{figure}[h!]
\centering
\begin{subfigure}{.5\textwidth}
  \centering
  \includegraphics[width=8.5cm, height=7cm]{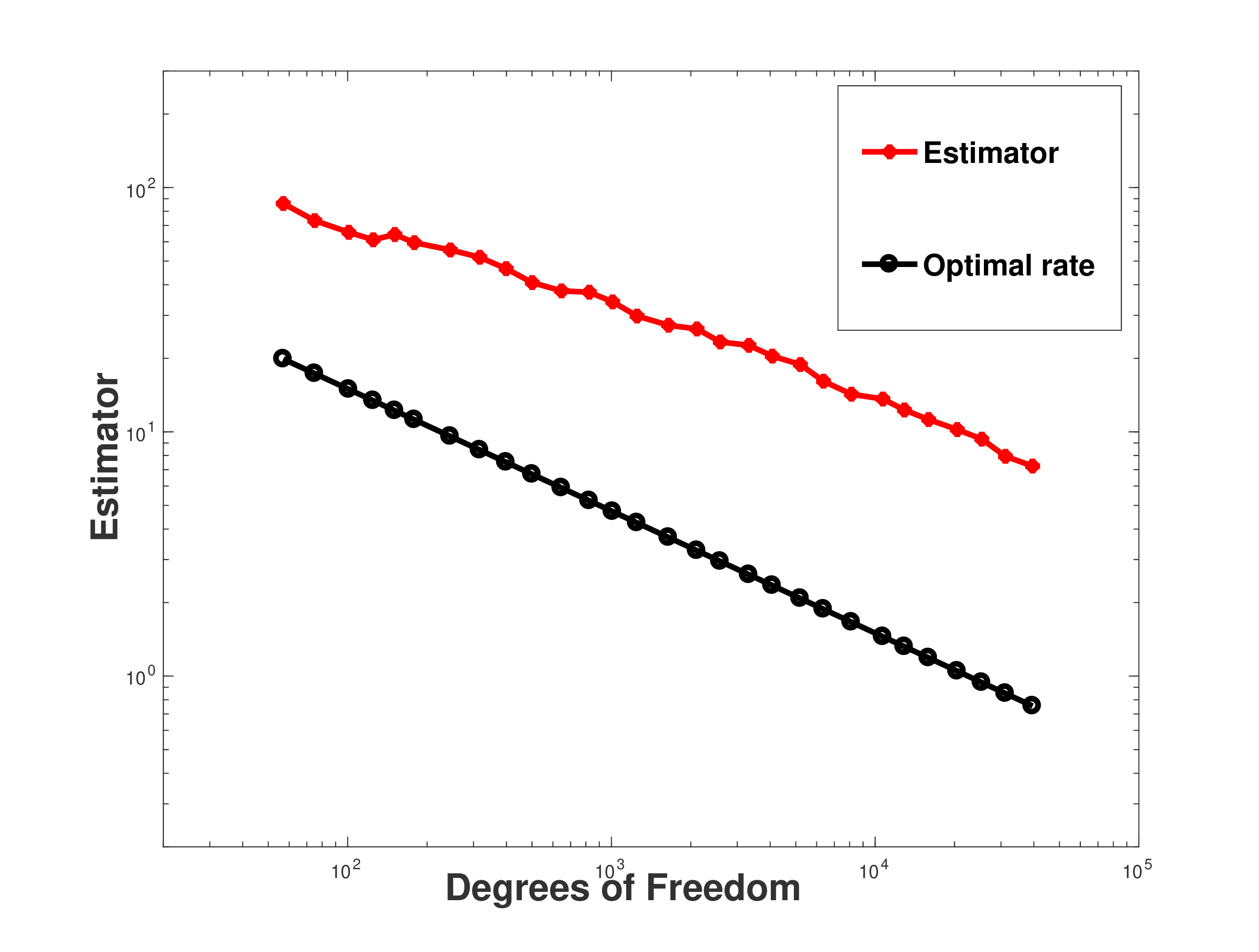}
  \caption{Estimator}\label{Ex_4:err}
\end{subfigure}%
\begin{subfigure}{.5\textwidth}
 \centering
 \includegraphics[width=8.5cm, height=7cm]{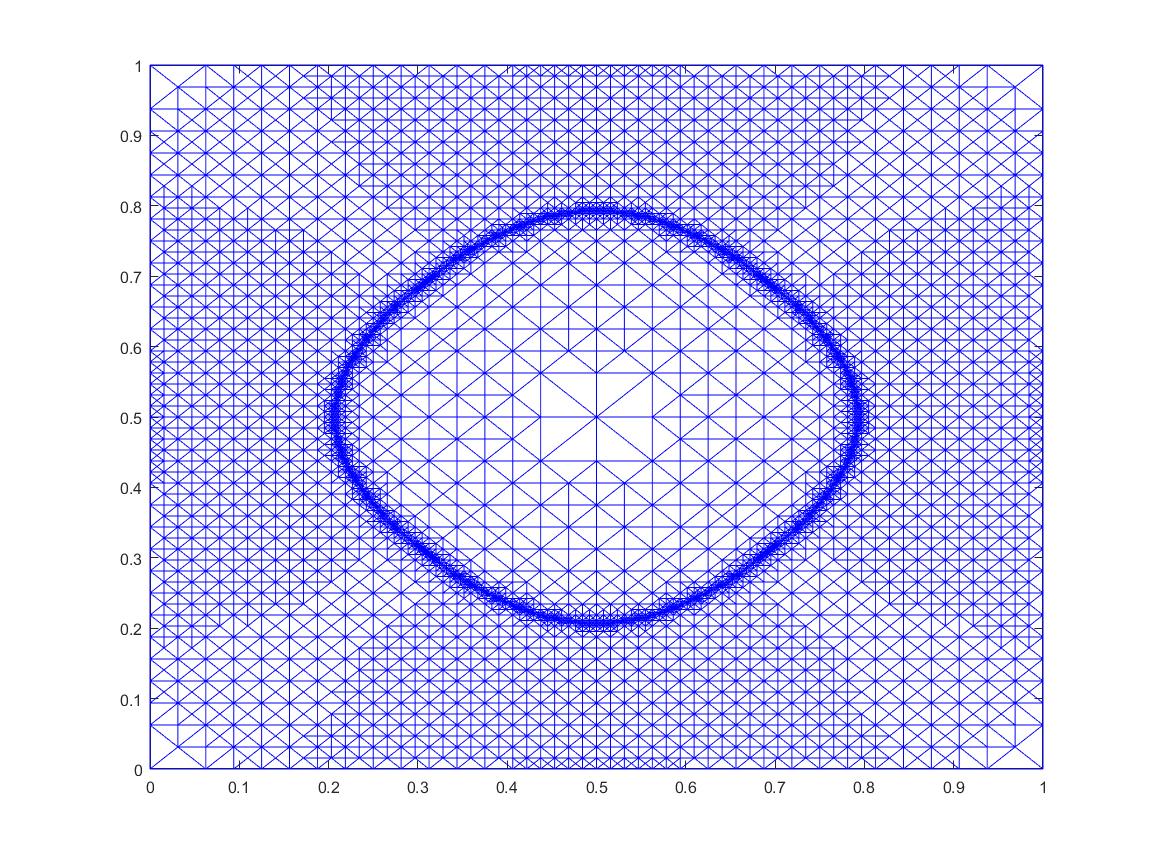}
   \caption{Adaptive mesh}\label{Ex_4:ada}
\end{subfigure}
\caption{Estimator and adaptive mesh for Example \ref{Exam:4} }
\end{figure}

%
%


\begin{thebibliography}{plain} {\small

\bibitem{adams}   { R. A. Adams and J. J. F.  Fournier},  {Sobolev Spaces},  {Pure and Applied Mathematics,  Elsevier, Amsterdam, 2003}.

\bibitem{ainsworth:2000} {M. Ainsworth and J. T. Oden,} {A Posteriori Error Estimation in Finite Element Analysis}, {Pure and Applied Mathematics, Wiley-Interscience, New York, 2000}. 

\bibitem{arnautu} {V. Arnautu and P. Neittaanm\"{a}ki}, { Discretization estimates for an elliptic control problem}, {Numer. Funct. Anal. Optim., 19(1998), pp. 431-464}.

\bibitem{becker:2000} {R. Becker, H. Kapp and R. Rannacher}, {Adaptive finite element methods for optimal control governed by partial differential equations: Basic concept}, {SIAM J. Control Optim., 39(2000), pp. 113-132}.



\bibitem{BIK:1999:PrimalDual}
{M. Bergounioux, K. Ito and K. Kunisch}, { Primal-dual strategy for constrained optimal control problems}, SIAM J. Control Optim., 37(1999), pp. 1176-1194.

 \bibitem{kunisch2002} {M. Bergounioux and K. Kunisch}, { Primal-dual strategy for state constrained optimal control problems}, {Comput. Optim. Appl., 22(2002), pp. 193-224}.
 
 \bibitem{kunisch2003} { M. Bergounioux and K. Kunisch}, {On the structure of Lagrange multipliers for state constrained optimal control problems}, {Systems Control Lett., 48(2003), pp. 169-176}. 
 
  \bibitem{brennerscott} { S. C. Brenner and L. R. Scott}, { The Mathematical Theory of Finite Element Methods}, {Springer-Verlag, New York, 2008}.
  


\bibitem{brenner:sung:2017} {S. C. Brenner and L. Y.  Sung}, {A new convergence analysis of finite element methods for elliptic distributed optimal control problems with pointwise state constraints}, {SIAM J. Control Optim., 55(2017), pp. 2289-2304}.

\bibitem{brenner:et:al:2019} { S. C. Brenner, L. Y. Sung and Y. Zhang}, {$C^0$ interior penalty methods for an elliptic state-constrained optimal control problem with Neumann boundary condition}, {J. Comput. Appl. Math., 350(2019), pp. 212-232}.
 
 \bibitem{BSZ:2013:C0IP}
 S. C. Brenner, L. Y. Sung and Y. Zhang,
 {A $C^0$ interior penalty method for an elliptic optimal control problem with state constraints},
 { In Recent Developments in Discontinuous
Galerkin Finite Element Methods for Partial Differential Equations (2012 John H.
Barrett Memorial Lectures), X. Feng, O. Karakashian and Y. Xing, ed., IMA Volumes
in Mathematics and Its Applications}, 157(2013), pp. 97-132.



%

\bibitem{BGKS:2018:SC}
S. C. Brenner, T. Gudi, K. Porwal and  L. Y.  Sung, A Morley finite element method for an elliptic distributed optimal control problem with pointwise state and control constraints, ESAIM Control Optim.  Calc. Var., 24(2018), pp. 1181-1206.

\bibitem{BWZ:2004:DPI} 
 S. C. Brenner, K. Wang and J. Zhao, Poincare-Friedrichs inequalities for piecewise H2 functions, Numer.
Funct. Anal. Optim., 25(2004), pp. 463-478.

 

 
\bibitem{Casas:1993:StateConst}
 E. Casas, { Control of an elliptic problem with pointwise state constraints,}{ SIAM J. Control Optim.}, 24(1986), pp. 1309-1318.
 
\bibitem{casas:finite} {E. Casas}, {Error estimates for the numerical approximation of semilinear elliptic control problems with finitely many state constraints}, {ESIAM Control Optim. Calc. Var., 8(2002), pp. 345-374}. 

\bibitem{casas:2012:sparse} { E. Casas, C. Clason and K. Kunisch},
{Approximation of elliptic control problems in measure spaces with sparse solutions}, {SIAM J. Control  Optim., 50(2012), pp. 1735-1752}.

\bibitem{clason:2011} {C. Clason and K. Kunisch}, {A duality-based approach to elliptic control problems in non reflexive Banach spaces}, {ESAIM Control Optim. Calc. Var., 17(2011), pp. 243-266}.
 
 \bibitem{casas:M} {E. Casas and M. Mateos}, {Uniform convergence of the FEM. Applications to state constrained control problems}, {Comput. Appl. Math., 21(2002), pp. 67-100}.
 
 
 \bibitem{casas:14} { E. Casas, M. Mateos and B. Vexler}, { New regularity results and improved error estimates for optimal control problems with state constraints}, {ESIAM Control Optim. Calc. Var., 20(2014), pp. 803-822}.
 

%
 
 \bibitem{chenz2019} {Y. Chen, J. Zhang, Y. Huang and Y. Xu}, {A posteriori error estimates of $hp$ spectral methods for integral state constrained elliptic optimal control problems} {Appl. Numer. Math., 144(2019), pp. 42-58}.

\bibitem{Ciarlet:1978:FEM}
 P. G. Ciarlet, The Finite Element Method for Elliptic Problems, North-Holland, Amsterdam, 1978.
 
%

\bibitem{DH:2007:StateCtrl} 
K. Deckelnick and M. Hinze,{ Numerical analysis of a control and state constrained elliptic control problem with piecewise constant control approximations,} { Proceeding of ENUMATH}, 2007, pp. 597-604.



\bibitem{Falk:1973:OCP}
R. Falk, {Approximation of a class of optimal control problems with order of convergence estimates,}{ J. Math. Anal. Appl.}, 44(1973), pp. 28-47.

\bibitem{ge2009} {L. Ge, W.Liu and D. Yang}, {Adaptive finite element approximation for a constrained optimal control problem via Multi-meshes}, {J. Sci. Comput., 41(2009), pp. 238-255}.

\bibitem{Geveci:1979:OCP}
T. Geveci, {On the approximation of the solution of an optimal control problem governed by an elliptic equation}, RIARO Anal. Num\'{e}r., 13 (1979), pp. 313-328.

\bibitem{Glowinsiki:1984:book}
R. Glowinski, Numerical Methods for Nonlinear Variational Problems, Springer-Verlag, New York,
1984.

\bibitem{gong:2017} {W. Gong and N. Yan}, {Adaptive finite element method for elliptic optimal control problems: convergence and optimality}, {Numer. Math., 135(2017), pp. 1121-1170}.



\bibitem{grisvard} { P. Grisvard}, {Elliptic Problems in Non Smooth Domains}, {Pitman, Boston, 1985}.
 
%

\bibitem{GH:2008:StateApost}
A. G\"unther and M. Hinze. { A posteriori error control of a state constrained elliptic control
problem.}{ J. Numer. Math.,} 16(2008), pp. 307-322.

\bibitem{hinze:2005} {M. Hinze}, {A variational discretization concept in control constrained optimization: the linear-quadratic case}, {Comput. Optim. Appl., 30(2005), pp. 45-61.} 



 
  




\bibitem{HPU:2009:Appls}
M. Hinze, R. Pinnau, M. Ulbrich and S. Ulbrich, { Optimization with PDE Constraints.}
{Math. Model. Theory Appl.,} 23, Springer, New York, 2009.



\bibitem{hoppe:2010} {R. H. W. Hoppe and M. Kieweg}, { Adaptive finite element methods for mixed control-state constrained  optimal control problems for elliptic boundary value problems}, {Comput. Optim. Appl., 46(2010), pp.  511-533}.

\bibitem{HK:2009:StateApost}
R. H. W. Hoppe and M. Kieweg. { A posteriori error estimation of finite element approximations
of pointwise state constrained distributed control problems.}{ J. Numer. Math.,} 17(2009), pp. 219-244.

\bibitem{ito} {  K. Ito and K. Kunisch}, {Lagrange Multiplier Approach to Variational Problems and Applications}, {Socity for Industrial and Applied Mathematics}, {Philadelphia, 2000}.
 
 \bibitem{kinderlehrer} { D. Kinderlehrer and G. Stampacchia}, { An Introduction to Variational Inequalities and Their Applications}, {SIAM, Philadelphia, 2000}.  
 
 \bibitem{kohls:2015} {K. Kohls, A. R\"{o}sch and K. G. Siebert}, {Convergence o adaptive finite elements for optimal control problems with control constraints}, {Internat. Ser. Numer. Math., 165(2015), pp. 403-419}.
 
 \bibitem{kornhuber} {R. Kornhuber}, {A posteriori error estimates for elliptic variational inequalities}, {Comput. Math. Appl. 31(1996), pp. 49-60}.

 \bibitem{lapin17} { A. V. Lapin and D. G. Zalyalov}, {Solution of Elliptic Optimal Control Problem with pointwise and non-local state constraints}, {Russian Mathematics, 61(2017), pp. 18-28}. 

\bibitem{LL:1975:NonConf}
P. Lascaux and  P. Lesaint, Some nonconforming finite elements for the plate bending problem, RAIRO Anal. Numer. R-1: 1975, pp. 9-53.

\bibitem{li2007} { R. Li, W. Liu and Y. Yan}, {A posteriori error estimates of recovery type for distributed convex optimal control problems}, {J. Sci. Comput., 33(2007), pp. 155-182}.

\bibitem{Lions:1971:OCP}
 J. L. Lions, Optimal Control of Systems Governed by Partial Differential Equations, Springer-Verlag, Berlin, 1971.
 
 \bibitem{liuyan2000} {W. Liu and N. Yan}, {A posteriori error estimators for a class of variational inequalities}. {J. Sci. Comput., 15(2000), pp. 361-393}. 
 
 \bibitem{liuyan2001} {W. Liu and N. Yan}, {A posteriori error estimates for distributed convex optimal control problems}, {Adv. Comput. Math., 15((2001), pp. 285-309}.
 
 \bibitem{liuyan2008} { W. Liu and N. Yan}, {Adaptive Finite Element Methods for Optimal Control Governed by PDEs,} {Scientific Press, Beijing, 2008}.
 
 \bibitem{LYG:2009:StateConst}
W. B. Liu, N. Yan and W. Gong, { A new finite element approximation of a state constrained
optimal control problem,} {J. Comput. Math.,} 27(2009), pp. 97-114.





\bibitem{LYYG:2010:IntState}
W. B. Liu, D. P. Yang, L. Yuan and C. Q. Gao, { Finite element approximations of an optimal control problem with integral state constraint,} {SIAM J. Numer. Anal., }48(2010), pp. 1163-1185. 

%

 \bibitem{luenberger} { D. G. Luenberger}, {Optimization by Vector Space Methods}, { John Wiley and Sons Inc., New York, 1969}.
 

\bibitem{Meyer:2008:StateConst}
C. Meyer, { Error estimates for the finite-element approximation of an elliptic control problem with pointwise state and control constraints,} {Control Cybern.,} 37(2008), pp. 51-83.


 
\bibitem{MRT:2005:Regular}
C. Meyer, A. R\"osch and F. Tr\"oltzsch, {Optimal control of PDEs with regularized pointwise
state constraints,} { Comp. Optim. and Appl.,} 33(2005), pp. 209-228.

\bibitem{Morley:1968}
L. S. D. Morley, The triangular equilibrium problem in the solution of plate bending problems, Aero. Quart. 19(1968), pp. 149-169.



\bibitem{NTW:2000:GMorley}
T. K. Nilssen, X.-C. Tai and R. Winther, {A robust nonconforming H2-element.}
Math. Comp., 70(2000), pp. 489-505.

\bibitem{PS:2021:APNUM}
K. Porwal and P. Shakya, A finite element method for an elliptic optimal control problem with integral state constraints. Appl. Numer. Math., 169(2021) pp. 273--288.

\bibitem{RW:2012:StateCtrl_Apost}
A. R\"osch and D. Wachsmuth, {A posteriori error estimates for optimal control problems with state and control constraints.}{ Numer. Math.,} 120(2012), pp. 733-762.


 \bibitem{pshakya:2019} {P. Shakya and R. K. Sinha}, {A priori and a posteriori error estimates of finite element
approximations for elliptic optimal control problem with
measure data},  {Optim. Control Appl. Meth., 40(2019), pp. 241-264}.


\bibitem{fredi10} { F. Tr\"{o}ltzsch}, { Optimal Control of Partial Differential  Equations}, {AMS Providence, RI, 2010}.
%




\bibitem{verfurth:1995} { R.  Verf\"{u}rth,} { A Review of A Posteriori Error Estimation and Adaptive Mesh-Refinement Techniques,} {Wiley-Teubner, Chichester, 1995}.

\bibitem{wolfmayr:2016} {M. Wolfmayr}, { A note on functional a posteriori estimates for elliptic optimal control problems}, {Numer. Methods Partial Differential Equations, 33(2016), pp. 403-424}.

\bibitem{yuan2009} { L. Yuan and D. Yang}, { A posteriori error estimate of optimal control problem of PDE with integral constraint for state}, {J. Comput. Math., 27(2009), pp. 525-542}.

\bibitem{ZY:2015:LGS}
J. Zhou and D. Yang,  Legendre-Galerkin spectral methods for optimal control problems with integral constraint for state in one dimension, Comput. Optim. Appl., 61(2015), pp. 135-158.

\bibitem{lixin17} { L. Zhou}, {A priori error estimates for optimal control problems with state and control constraints}, {Optim Control Appl. Meth., 39(2018), pp.1168-1181}.
 

 }
\end{thebibliography}
\end{document}